\documentclass[11pt]{article}

\usepackage[a4paper,margin=0.92in]{geometry}
\usepackage[T1]{fontenc}
\usepackage{amsmath,amssymb,amsthm,mathtools,bm}
\usepackage{lmodern}
\usepackage{mathrsfs}
\usepackage{microtype}
\usepackage{booktabs,threeparttable,array,multirow,makecell,tabularx}
\usepackage{graphicx}
\usepackage{subcaption}
\usepackage[ruled,vlined,linesnumbered]{algorithm2e}
\usepackage{enumitem}
\usepackage{siunitx}
\usepackage[numbers,sort&compress]{natbib}
\usepackage{hyperref}
\usepackage{aliascnt}
\usepackage[nameinlink,capitalize,noabbrev]{cleveref}
\usepackage{xcolor}
\usepackage{xspace}
\usepackage{url}
\usepackage{placeins}
\usepackage{float}

\graphicspath{{figures/}}

\hypersetup{
  colorlinks=false,
  pdfborder={0 0 1},
  linkbordercolor={0 0 1},
  citebordercolor={0 0.5 0},
  urlbordercolor={0 0 1},
  pdftitle={Self-Consistent Adjoint Policy Iteration for Constrained Dynamic Portfolio Choice},
  pdfauthor={Jeonggyu Huh, Yeoneung Kim, Seungwon Jeong}
}

\newtheorem{theorem}{Theorem}[section]

\newaliascnt{proposition}{theorem}
\newtheorem{proposition}[proposition]{Proposition}
\aliascntresetthe{proposition}

\newaliascnt{lemma}{theorem}
\newtheorem{lemma}[lemma]{Lemma}
\aliascntresetthe{lemma}

\newaliascnt{corollary}{theorem}
\newtheorem{corollary}[corollary]{Corollary}
\aliascntresetthe{corollary}

\newaliascnt{assumption}{theorem}
\newtheorem{assumption}[assumption]{Assumption}
\aliascntresetthe{assumption}

\theoremstyle{remark}
\newtheorem{remark}{Remark}[section]

\crefname{assumption}{Assumption}{Assumptions}
\Crefname{assumption}{Assumption}{Assumptions}
\crefname{theorem}{Theorem}{Theorems}
\Crefname{theorem}{Theorem}{Theorems}
\crefname{proposition}{Proposition}{Propositions}
\Crefname{proposition}{Proposition}{Propositions}
\crefname{lemma}{Lemma}{Lemmas}
\Crefname{lemma}{Lemma}{Lemmas}
\crefname{corollary}{Corollary}{Corollaries}
\Crefname{corollary}{Corollary}{Corollaries}
\crefname{algorithm}{Algorithm}{Algorithms}
\Crefname{algorithm}{Algorithm}{Algorithms}

\newcommand{\R}{\mathbb{R}}
\newcommand{\E}{\mathbb{E}}
\newcommand{\Pp}{\mathbb{P}}
\newcommand{\F}{\mathcal{F}}
\newcommand{\U}{\mathcal{U}}
\newcommand{\Qcal}{\mathcal{Q}}
\newcommand{\Lcal}{\mathcal{L}}
\newcommand{\Tcal}{\mathcal{T}}
\newcommand{\Acal}{\mathcal{A}}
\newcommand{\AcalNum}{\widetilde{\mathcal{A}}}
\newcommand{\ResOp}{\mathcal{R}}
\newcommand{\Hgen}{\mathfrak{H}}
\newcommand{\EvalOp}{\operatorname{Eval}}
\newcommand{\FitOp}{\operatorname{Fit}}
\newcommand{\RecOp}{\operatorname{Rec}}
\newcommand{\InterpOp}{\operatorname{Interp}}
\newcommand{\Kset}{\mathcal{K}}
\newcommand{\PortSet}{\mathcal{C}_{\mathrm{port}}}
\newcommand{\Hpol}{\mathcal{H}_{\mathrm{pol}}}
\newcommand{\D}{\mathrm{D}}
\newcommand{\dd}{\,\mathrm{d}}
\newcommand{\tr}{\operatorname{tr}}
\newcommand{\argmax}{\operatorname*{arg\,max}}
\newcommand{\argmin}{\operatorname*{arg\,min}}

\newcommand{\ipgdpo}{\textsc{I-PGDPO}\xspace}
\newcommand{\norm}[1]{\left\lVert #1\right\rVert}

\newcommand{\inner}[2]{\left\langle #1,#2\right\rangle}
\newcommand{\epscons}{\varepsilon_{\mathrm{cons}}}
\newcommand{\epsdisc}{\varepsilon_{\mathrm{disc}}}
\newcommand{\epsstat}{\varepsilon_{\mathrm{stat}}}
\newcommand{\epsrepr}{\varepsilon_{\mathrm{repr}}}
\newcommand{\epsrec}{\varepsilon_{\mathrm{rec}}}
\newcommand{\epsext}{\varepsilon_{\mathrm{ext}}}
\newcommand{\epsham}{\varepsilon_{\mathrm{ham}}}
\newcommand{\Adv}{\mathfrak{A}}

\title{\textbf{Self-Consistent Adjoint Policy Iteration for}\\
\textbf{Constrained Dynamic Portfolio Choice}}

\author{
Jeonggyu Huh\textsuperscript{1}\thanks{Corresponding author. Address for proofs: Department of Mathematics, Sungkyunkwan University, Suwon 16419, Republic of Korea. E-mail: \href{mailto:jghuh@skku.edu}{\texttt{jghuh@skku.edu}}}\quad
Yeoneung Kim\textsuperscript{2}\quad
Seungwon Jeong\textsuperscript{3}\\[0.50em]
\small \textsuperscript{1}Department of Mathematics, Sungkyunkwan University, Suwon 16419, Republic of Korea\\
\small \textsuperscript{2}Department of Applied Artificial Intelligence,\\
\small Seoul National University of Science and Technology, Seoul 01811, Republic of Korea\\
\small \textsuperscript{3}Global-Learning \& Academic Research Institution for Master's and PhD Students, and Postdocs,\\
\small Chonnam National University, Gwangju 61186, Republic of Korea
}
\date{}

\begin{document}
\maketitle

\begin{abstract}
We develop simulation-based policy iteration for continuous-time portfolio choice with predictable returns and convex constraints. Each outer step re-evaluates a fixed-latent open-loop backpropagation-through-time (OL-BPTT) adjoint after deployment and solves the constrained update. Shifted-adjoint cancellation controls the adjoint--HJB Hamiltonian-gradient discrepancy by the policy-improvement residual. For CRRA portfolios, exact HJB policy iteration identifies the optimal reduced value factor, while population OL-BPTT iteration converges globally when the adjoint update is directionally improving and approximate stationarity is asymptotically HJB-compatible. A theorem-matched occupation audit yields maximal $95\%$ upper endpoints of $0.066$ for the primitive directional ratio and $0.074$ for a stronger norm-relative ratio, both against the half-step threshold $0.75$. In the high-precision $50$--$50$ occupancy/broad-anchor design of a three-factor, fifty-asset benchmark, current-policy re-evaluation outperforms matched pooled refinement under the on-policy and broad evaluation laws.
\end{abstract}




\noindent\textbf{Running title:} Adjoint policy iteration for portfolio choice.

\section{Introduction}
\label{sec:intro}

Continuous-time portfolio choice with stochastic investment opportunities is analytically tractable only in a limited set of unconstrained models. Classical results characterize optimal investment through dynamic programming, martingale duality, or explicit affine solutions \citep{Merton1969,CvitanicKaratzas1992,XuShreve1992b,Zariphopoulou1994,KimOmberg1996,Liu2007}. Once short-sale, borrowing, or asset-wise position limits are imposed, the portfolio rule develops active regions and switching boundaries. With several return-predicting factors, a direct Hamilton--Jacobi--Bellman (HJB) discretization also becomes expensive even though the portfolio action itself remains a convex quadratic problem.

Simulation-based dynamic portfolio methods avoid a full factor grid \citep{BrandtEtAl2005}, while recent financial algorithms combine stochastic approximation or reinforcement learning with convergence and error analysis \citep{CostaGadatHuang2025,WangGaoLi2026}. Neural HJB and backward stochastic differential equation (BSDE) methods, together with continuous-time learning methods, approximate global value, $q$-function, actor--critic, or policy-gradient objects \citep{HanJentzenE2018,HurePhamBachouchLangrene2021,WangZariphopoulouZhou2020,JiaZhouPolicyEvaluation2022,ReisingerStockingerZhang2023}. Their bottleneck is often a repeatedly trained global approximator. We instead retain only the low-dimensional continuation statistic needed for the dynamic-programming improvement step and solve the high-dimensional constrained portfolio block explicitly.

Neural policy-iteration implementations can replace one nonlinear HJB equation by a sequence of fixed-policy equations, but still solve a global neural approximation problem at successive evaluations \citep{ItoReisingerZhang2021,MengEtAl2024PINNPI}. Two iterative traditions provide the closest context. Stochastic methods of successive approximations (MSA) alternate the forward state, adjoints from the Pontryagin maximum principle (PMP), and Hamiltonian control updates, including settings with control in the diffusion coefficient and second-order adjoints \citep{KerimkulovSiskaSzpruch2021,JiXu2022MSA}. HJB policy iteration instead alternates fixed-policy value equations and greedy feedback improvement \citep{ItoReisingerZhang2021}. Our construction retains current-policy simulation and adjoint re-evaluation from MSA, but turns the resulting conditional adjoint information into a Markov feedback-improvement map on state--time restart points and compares that map directly with exact fixed-policy HJB improvement. Thus, after a flexible warm start, the outer loop uses fresh simulation and conditional field estimation without requiring repeated training of a global policy or value network.

For constant relative risk aversion (CRRA) utility with risk aversion $\gamma>1$, a fixed portfolio feedback $u$ has value $V^u$ at time $t$, wealth $x$, and factor state $z$, with reduced value factor $F^u$:
\[
  V^u(t,x,z)=\frac{x^{1-\gamma}}{1-\gamma}F^u(t,z).
\]
The local portfolio improvement depends on the normalized hedging field $R^u=\D_z\log F^u$. With excess returns $\mu(z)$, return covariance $\Sigma$, return--factor cross-covariance $C$, and feasible portfolio set $\PortSet$, it is the $\Sigma$-metric projection
\[
  \Pi_{\PortSet}^{\Sigma}\!\left(
  \frac1\gamma\Sigma^{-1}\bigl[\mu(z)+C R^u(t,z)\bigr]
  \right).
\]
Thus a many-asset problem reduces to estimating a factor-dependent field and then solving an explicit metric projection or quadratic program (QP). Our population adjoint-based operator replaces $R^u$ by the fixed-latent open-loop backpropagation-through-time (OL-BPTT) target $R_{\mathrm{OL}}^u$. Starting from a feasible feedback $u^k$, each outer step simulates under the deployed policy, reconstructs $R_{\mathrm{OL}}^{u^k}$ over state--time restart points, computes the constrained update, and deploys a damped policy.

Re-evaluation is economically relevant because the implemented portfolio changes both the state law and the intertemporal hedging field. A second estimate under the initial policy is therefore not generally an estimate of the same policy-improvement operator as a current-policy update. The distinction disappears in the Merton benchmark, where the operator is policy independent, but it is material under predictable returns and binding constraints. This motivates both the outer iteration and the matched-budget comparison between current-policy re-evaluation and pooled refinement under the initial policy.

\paragraph{Relation to earlier and parallel work.}
The terminology and one-shot motivation begin with Pontryagin-guided direct policy optimization \citep{HuhEtAl2025}; no result from that preprint is used in the analysis below. That work uses adjoint information to recover a corrected control from a trained policy, but does not analyze repeated deployment and current-policy re-evaluation. Jeon et al.~\citep{JeonHuhKooLim2026} is an independent manuscript. It develops a broader one-shot adjoint-to-control framework, including general moving feasible fibers, first- and second-order stochastic maximum-principle adjoints, and shifted-martingale acquisition. The present paper addresses the subsequent deployed recursion rather than the one-step acquisition problem treated in those works. No theorem or proof in the present paper invokes a result from that manuscript. Section~\ref{sec:framework} defines the population map $\Acal$ and its finite approximation $\AcalNum_k$ directly, while Section~\ref{sec:portfolio-decoder} defines $R_{\mathrm{OL}}^u$ and the explicit CRRA portfolio update. Writing $\vartheta^u$ for the required population adjoint field, the new object is the deployed recursion
\[
  u^0\longrightarrow \vartheta^{u^0}\longrightarrow u^1
  \longrightarrow \vartheta^{u^1}\longrightarrow u^2\longrightarrow\cdots.
\]

The three-operator hierarchy
\[
  \AcalNum_k\longrightarrow\Acal\longrightarrow\Tcal
\]
separates finite computation, population adjoint-based policy improvement, and exact HJB policy improvement. At the generic level, the main obstruction is the second-order term created by control-dependent diffusion. For state $s$, second adjoint $P^u$, and diffusion coefficient $\sigma$, shifted-adjoint cancellation confines $P^u-\D^2_{ss}V^u$ to
\[
  \bigl(P^u-\D^2_{ss}V^u\bigr)
  \bigl(\sigma(a)-\sigma(u)\bigr),
\]
so the second-adjoint discrepancy vanishes at the current action and is scaled by the updated-control diffusion displacement. In the CRRA reduction, the entire population Hamiltonian-gradient discrepancy collapses to
\[
  C\bigl(R_{\mathrm{OL}}^u-R^u\bigr).
\]

The main contributions are:
\begin{enumerate}[leftmargin=2.1em,itemsep=0.25em]
  \item We close the one-shot fixed-latent adjoint-to-control map into a self-consistent constrained policy iteration with damping and an interpolate-then-project output policy.
  \item On fixed-stratum regular future tubes, shifted-adjoint cancellation and fixed-latent envelope analysis control the adjoint--HJB Hamiltonian-gradient discrepancy by the HJB policy-improvement residual, without requiring $P^u\to\D^2V^u$.
  \item For quadratic HJB action Hamiltonians, we prove exact and approximate value-improvement bounds. For constrained CRRA portfolios, exact HJB policy iteration and population OL-BPTT iteration converge globally, the latter under directional occupation alignment and asymptotic stationarity compatibility.
  \item The numerical study identifies when re-evaluation is useful and audits both the primitive directional condition and a stronger norm-relative condition. Across $180$ theorem-matched path banks, the largest $95\%$ upper endpoints are $0.066341$ for $\widehat\kappa_{\mathrm{dir}}$ and $0.073936$ for $\widehat\kappa_{\mathrm{occ}}$, compared with the sufficient half-step threshold $0.75$. In the high-precision $50$--$50$ occupancy/broad-anchor design of the three-factor, fifty-asset benchmark, current-policy re-evaluation beats matched pooled initial-policy refinement under the on-policy and broad evaluation laws in all three seeds.
\end{enumerate}

The generic consistency theorem is local to regular interior or fixed-face tubes and does not itself verify the global directional-alignment or stationarity-compatibility hypotheses. The primitive directional condition is audited directly; a stronger norm-relative condition, also audited on the visited benchmark sequence, supplies both hypotheses. Global convergence is proved for the constrained CRRA population subclass. Finite sampled implementations remain governed by the decomposed a posteriori error certificate.

\section{Self-consistent adjoint policy iteration}
\label{sec:framework}

\subsection{Controlled diffusion and frozen warm start}

Let $(\Omega,\F,(\F_t)_{0\le t\le T},\Pp)$ support a $d_W$-dimensional Brownian motion $W=(W_t)_{0\le t\le T}$. The controlled state $S_t\in\R^{d_S}$ satisfies
\begin{equation}
  \dd S_t=b(t,S_t,u_t)\dd t+\sigma(t,S_t,u_t)\dd W_t,
  \qquad S_0=s_0,
  \label{eq:sde}
\end{equation}
with progressively measurable control
\begin{equation}
  u_t\in\U(t,S_t)\subseteq\R^{d_u}.
  \label{eq:constraint}
\end{equation}
The objective is
\begin{equation}
  J(u)=\E\left[\int_0^T \ell(t,S_t,u_t)\dd t+g(S_T)\right].
  \label{eq:objective}
\end{equation}
The control problem is to maximize $J(u)$ over admissible controls. We write $S^u$ for the state under feedback $u$, and $\E_{t,s}^u$ for expectation when the controlled process is restarted from $S_t^u=s$ and then follows $u$.
A neural or otherwise flexible warm start can enforce feasibility through a local or global chart
\begin{equation}
  u_\phi(t,s)=\Psi(t,s,a_\phi(t,s)).
  \label{eq:chart}
\end{equation}
After training, $\phi$ is frozen and the resulting feasible policy is denoted by $u^0$. Neural-network evaluation may therefore remain inside a rollout, but no later theorem depends on convergence of the warm-start parameter optimizer.

\subsection{First and second adjoints for policy improvement}
\label{sec:adjoint-recovery}

Fix a rollout policy $u$ and a regular feasibility chart $u_t=\Psi(t,S_t,a_t)$. \emph{Fixed-latent} (frozen-selector) differentiation means that the realized selector $a_t$ is held fixed while the chart-composed coefficients are differentiated; structural state derivatives of $\Psi$ are therefore retained. In the state-independent action charts used by the iterative benchmarks below, those chart-composed derivatives reduce to the primitive state derivatives. The adjoint--HJB consistency theorem in \cref{sec:compatibility-theory} is stated in these direct action coordinates; a moving-fiber extension would replace primitive derivatives by fixed-latent derivatives of the chart-composed coefficients and is not claimed here.

Subscripts $s$ and $a$ denote state and action derivatives, and $\sigma^{(j)}$ is the $j$th diffusion column. Let $(\lambda^u,\mathsf Z^u)$ be the vector first adjoint and its Brownian martingale coefficient. In direct action coordinates, with every state derivative taken at fixed action $u_t$, set
\begin{equation}
  A_t^u:=b_s(t,S_t^u,u_t),\qquad
  C_{j,t}^u:=\sigma_s^{(j)}(t,S_t^u,u_t),
  \label{eq:adjoint-coefficient-definitions}
\end{equation}
and write $\mathsf Z_t^{u,(j)}$ for the $j$th column of $\mathsf Z_t^u$. The first-adjoint BSDE is
\begin{equation}
\begin{split}
  -\dd\lambda_t^u={}&\left[
  \ell_s(t,S_t^u,u_t)+(A_t^u)^\top\lambda_t^u
  +\sum_{j=1}^{d_W}(C_{j,t}^u)^\top\mathsf Z_t^{u,(j)}
  \right]\dd t-\mathsf Z_t^u\dd W_t,\\
  \lambda_T^u={}&\D g(S_T^u).
\end{split}
\label{eq:first-adjoint-bsde}
\end{equation}
Let $P^u\in\mathbb S^{d_S}$ be the matrix second adjoint and $\mathsf R^{u,(j)}\in\mathbb S^{d_S}$ its martingale coefficients, where $\mathbb S^{d_S}$ is the space of symmetric matrices. Using the Frobenius inner product $\inner{A}{B}_F:=\tr(A^\top B)$, set
\begin{equation}
  \mathcal F_{P,t}^u
  :=\D_{ss}^2\!\left[
  \ell+\lambda_t^{u\top}b+\inner{\mathsf Z_t^u}{\sigma}_F
  \right](t,S_t^u,u_t),
  \label{eq:second-adjoint-driver-definition}
\end{equation}
where the adjoint inputs are frozen and the state derivative again holds the action fixed, the second-adjoint BSDE is
\begin{equation}
\begin{split}
  -\dd P_t^u={}&\Bigg[\mathcal F_{P,t}^u+(A_t^u)^\top P_t^u+P_t^uA_t^u
  +\sum_{j=1}^{d_W}(C_{j,t}^u)^\top P_t^uC_{j,t}^u\\
  &\qquad+\sum_{j=1}^{d_W}\Bigl[
  (C_{j,t}^u)^\top\mathsf R_t^{u,(j)}
  +\mathsf R_t^{u,(j)}C_{j,t}^u\Bigr]\Bigg]\dd t
  -\sum_{j=1}^{d_W}\mathsf R_t^{u,(j)}\dd W_t^{(j)},\\
  P_T^u={}&\D^2g(S_T^u).
\end{split}
\label{eq:second-adjoint-bsde}
\end{equation}
These are the fixed-latent stochastic maximum-principle adjoints \citep{Peng1990,YongZhou1999}. Their dimensions are $\lambda^u\in\R^{d_S}$, $\mathsf Z^u\in\R^{d_S\times d_W}$, and $P^u\in\mathbb S^{d_S}$. The matrix $P^u$ is not identified with a value Hessian. When the control enters the diffusion, finite action comparison uses the shifted martingale input
\begin{equation}
  \zeta_t^u:=\mathsf Z_t^u-P_t^u\sigma(t,S_t^u,u_t)
  \label{eq:shifted-adjoint}
\end{equation}
and the generalized Hamiltonian
\begin{equation}
  \Hgen(t,s;a,\lambda,\zeta,P)
  =\ell(t,s,a)+\lambda^\top b(t,s,a)
  +\inner{\zeta}{\sigma(t,s,a)}_F
  +\frac12\tr\!\left[\sigma(t,s,a)^\top P\sigma(t,s,a)\right].
  \label{eq:generalized-hamiltonian}
\end{equation}
Write $\vartheta^u=(\lambda^u,\zeta^u,P^u)$, retaining only the blocks required by the pointwise control problem. A \emph{state--time restart point} (called an anchor in the implementation) is a prescribed pair $(t,s)$ at which the state is restarted from $s$ and continuation paths are simulated under the current feedback. The resulting restarted-state quantity is the conditional adjoint variable estimated from those continuations; when a Markov version exists, it coincides with $\vartheta^u(t,s)$. The population-level adjoint-based policy-improvement operator is
\begin{equation}
  \Acal(u)(t,s)\in
  \argmax_{a\in\U(t,s)}\Hgen(t,s;a,\vartheta^u(t,s)).
  \label{eq:adjoint-recovery-map}
\end{equation}

\Cref{eq:first-adjoint-bsde,eq:second-adjoint-bsde,eq:shifted-adjoint,eq:generalized-hamiltonian,eq:adjoint-recovery-map} fully define the population one-step operator used throughout this paper. The present analysis concerns what happens when the recovered feedback is deployed and that population operation is re-evaluated; the broader one-shot acquisition framework is studied independently in Jeon et al.~\citep{JeonHuhKooLim2026}.
\subsection{HJB policy improvement, smooth PMP--HJB identification, and consistency}
\label{sec:compatibility}

For a fixed Markov feedback $u$, let $V^u(t,s)$ be its continuation value. When $V^u$ is smooth, the fixed-policy HJB action Hamiltonian is
\begin{equation}
  \Qcal^u(t,s,a)
  =\ell(t,s,a)+\D_sV^u(t,s)^\top b(t,s,a)
  +\frac12\tr\!\left[\sigma\sigma^\top(t,s,a)\D^2_{ss}V^u(t,s)\right].
  \label{eq:markov-score}
\end{equation}
The exact HJB policy-improvement operator is
\begin{equation}
  \Tcal(u)(t,s)\in\argmax_{a\in\U(t,s)}\Qcal^u(t,s,a).
  \label{eq:exact-map}
\end{equation}
Write $\ResOp(u):=\Tcal(u)-u$ for its policy-improvement residual. The exact value-improvement and global-convergence results below concern this operator. Throughout, $\Acal(u)$ and $\Tcal(u)$ denote fixed measurable selectors of the displayed argmax correspondences; under the strongly concave quadratic Hamiltonians used below, the selector is unique. The adjoint-based operator $\Acal(u)$ is a computational realization whose consistency with $\Tcal(u)$ must be established rather than assumed.

The population operator $\Acal(u)$ denotes the infinite-sample fixed-latent adjoint-based policy-improvement operator. In a generic diffusion its population adjoint target is the fixed-latent first/second-adjoint tuple defined in \cref{eq:shifted-adjoint,eq:generalized-hamiltonian}; in the CRRA realization it is the reduced target defined in \cref{sec:portfolio-decoder}. Its finite-computation approximation $\AcalNum_k(u)$ is obtained by reverse-mode automatic differentiation through a time-discretized rollout---that is, fixed-latent open-loop backpropagation through time (OL-BPTT)---followed by conditional projection or regression over state--time restart points and an exact convex QP update. Under consistent time discretization, sampling, representation, continuation extension, and QP-solver tolerance, this discrete estimator targets $\Acal(u)$; those numerical errors belong to the link $\AcalNum_k\to\Acal$. In a structure-reduced model, the same fixed-latent reverse pass is applied to the reduced fixed-policy representation and estimates only the field required by the reduced policy-improvement formula. By contrast, $\Tcal(u)$ is the dynamic-programming/HJB improvement defined from the exact fixed-policy continuation value. Thus $\AcalNum_k\to\Acal\to\Tcal$ separates finite discrete computation, population-level OL-BPTT/PMP adjoint evaluation, and the exact HJB reference. Accordingly, the population-level consistency, value-improvement, and convergence statements below depend on the adjoint-defined operator $\Acal$, not on any particular numerical realization $\AcalNum_k$.

Let $u^\star$ be a sufficiently smooth optimal Markov selector on a regular stratum, let $S^\star$ be its controlled state, and let $V$ be the corresponding smooth optimal value. Write $\sigma_t^\star:=\sigma(t,S_t^\star,u_t^\star)$. The standard maximum-principle/dynamic-programming identification \citep{Peng1990,YongZhou1999} is
\begin{equation}
  \lambda_t^\star=\D_sV(t,S_t^\star),\qquad
  \mathsf Z_t^\star=\D^2_{ss}V(t,S_t^\star)\sigma_t^\star,\qquad
  \zeta_t^\star=\bigl(\D^2_{ss}V(t,S_t^\star)-P_t^\star\bigr)\sigma_t^\star.
  \label{eq:smooth-adjoint-identification}
\end{equation}
Equation~\eqref{eq:smooth-adjoint-identification} makes the adjoint and HJB action gradients agree at the current optimal action, without requiring $P^\star=\D^2V$ or $\zeta^\star=0$; see Proposition~\ref{prop:shifted-gradient-cancellation}. Away from stationarity, the fixed-policy action gradients can differ.

Whenever the two Hamiltonians admit a common positive normalization, write $q_{\mathrm A}^u$ for the normalized adjoint Hamiltonian and $q_{\mathrm{HJB}}^u$ for the normalized HJB action Hamiltonian. Define the pointwise adjoint--HJB Hamiltonian-gradient discrepancy by
\begin{equation}
  e_{\mathrm{cons}}^u(t,s,a)
  :=\nabla_a q_{\mathrm A}^u(t,s,a)-\nabla_a q_{\mathrm{HJB}}^u(t,s,a).
  \label{eq:compatibility-error}
\end{equation}
For a positive-definite curvature matrix $H$, use $\norm{x}_H^2:=x^\top Hx$ and the dual norm $\norm{y}_{H^{-1}}^2:=y^\top H^{-1}y$. This discrepancy includes errors in both the linear and quadratic action coefficients and does not require the adjoint Hamiltonian to be represented through a value function.

\begin{remark}[Envelope cancellation]
In the direct action coordinates used by \cref{sec:compatibility-theory}, an interior stationary policy has vanishing action gradient:
\[
  \nabla_a q_{\mathrm{HJB}}^u(t,s,u(t,s))=0.
\]
Hence differentiating through the feedback contributes no first-order chain term. On a fixed regular active face, Karush--Kuhn--Tucker (KKT) stationarity places the HJB Hamiltonian gradient in the normal cone while $\D_su$ is tangent, giving the same cancellation. For an interior quadratic HJB Hamiltonian with action-independent term $c_0^u$, linear coefficient $m^u$, and curvature $H^u\succ0$,
\[
  q_{\mathrm{HJB}}^u(a)=c_0^u+(m^u)^\top a-\frac12a^\top H^ua,
\]
one has
\begin{equation}
  (\D_su)^\top\nabla_a q_{\mathrm{HJB}}^u(u)
  =(\D_su)^\top H^u\bigl(\Tcal(u)-u\bigr).
  \label{eq:envelope-policy-residual}
\end{equation}
Thus open-loop and closed-loop first variations agree at a fixed point of the HJB policy-improvement operator; during intermediate iterations their discrepancy is naturally tied to the policy-improvement residual and is retained in the consistency error $e_{\mathrm{cons}}^u$. Section~\ref{sec:compatibility-theory} turns this identity into a residual bound for the first-order fixed-latent OL-BPTT target on smooth interior or identified fixed-face regions. It still does not by itself bound the complete adjoint--HJB Hamiltonian-gradient discrepancy away from stationarity.
\end{remark}

\subsection{Numerical realization of the adjoint-based policy-improvement operator}

The scheme, denoted \ipgdpo, iterates policies and policy-improvement fields rather than neural parameters. The evaluator $\EvalOp_{h,M}^{\mathrm{OL}}(u)$ applies fixed-latent OL-BPTT through each discretized rollout and conditionally estimates the adjoint variable required by the pointwise policy-improvement problem over state--time restart points. In a generic diffusion this statistic contains the required first- and second-adjoint blocks; in a structure-reduced model the same fixed-latent reverse pass estimates only the reduced field sufficient for policy improvement. Here $h$ is the time-discretization step and $M$ is the path-equivalent simulation budget, and $\FitOp_{\mathfrak F}$ reconstructs the conditional field by a table, spline, structured basis, low-rank expansion, or fixed nonlinear feature map. The operator $\InterpOp$ denotes the continuous interpolation or regression applied to an averaged pre-projection control field before its final feasibility projection. In the reported experiments, all post-warm-start fits are globally solved linear or ridge problems; any error from a more general approximator enters the same a posteriori certificate. We write $\EvalOp_{h,M}:=\EvalOp_{h,M}^{\mathrm{OL}}$, and the represented adjoint field is
\begin{equation}
  \widehat\vartheta^k
  =\FitOp_{\mathfrak F_k}\bigl[\EvalOp_{h_k,M_k}(u^k)\bigr].
  \label{eq:represented-field}
\end{equation}
In the portfolio models studied here, $\RecOp$ returns the unconstrained maximizer of the estimated strongly concave quadratic action Hamiltonian, yielding the pre-projection control field:
\begin{equation}
  a_{\mathrm{raw}}^k(t,s)
  =\RecOp\bigl(t,s,\widehat\vartheta^k(t,s)\bigr),
  \label{eq:raw-recovery}
\end{equation}
Here $\Pi_{\U}$ denotes pointwise projection onto $\U(t,s)$ in the quadratic action-curvature metric, suppressed in the notation. This projection solves the constrained QP; in the CRRA case the metric is $\Sigma$, as in \cref{eq:portfolio-olbptt-recovery}, and it reduces to componentwise clipping in the reported diagonal-covariance box benchmarks. The computed policy update is
\begin{equation}
  \AcalNum_k(u^k):=\Pi_{\U}(a_{\mathrm{raw}}^k),
  \qquad
  \widetilde u^k:=\AcalNum_k(u^k).
  \label{eq:computed-map}
\end{equation}
The rollout policy is updated by
\begin{equation}
  u^{k+1}=(1-\beta_k)u^k+\beta_k\widetilde u^k,
  \qquad 0<\beta_k\le1.
  \label{eq:damped-update}
\end{equation}
The gain stabilizes the rollout policy that generates the next evaluation law, but it does not remove systematic adjoint, discretization, or representation bias.

\begin{algorithm}[!htbp]
\caption{\ipgdpo: self-consistent fixed-latent OL-BPTT policy iteration}
\label{alg:ipgdpo}
\DontPrintSemicolon
\KwIn{feasible warm start $u^0$; outer iterations $N_{\mathrm{out}}$; fixed-latent OL-BPTT evaluators $\EvalOp_{h_k,M_k}$; representation classes $\mathfrak F_k$; unconstrained recovery map $\RecOp$; gains $\beta_k$.}
\For{$k=0,1,\ldots,N_{\mathrm{out}}-1$}{
  Simulate the current rollout policy $u^k$\;
  Apply fixed-latent OL-BPTT to the generic adjoint tuple or the structure-reduced target, and conditionally estimate the required adjoint variable at the restart points\;
  Fit the adjoint field required by the policy update $\widehat\vartheta^k=\FitOp_{\mathfrak F_k}[\EvalOp_{h_k,M_k}(u^k)]$\;
  Form the reduced quadratic Hamiltonian and compute its unconstrained maximizer $a_{\rm raw}^k=\RecOp(\widehat\vartheta^k)$\;
  Solve the constrained QP by metric projection $\widetilde u^k=\AcalNum_k(u^k)=\Pi_{\U}(a_{\rm raw}^k)$, then deploy $u^{k+1}=(1-\beta_k)u^k+\beta_k\widetilde u^k$\;
  Monitor the Hamiltonian-gradient, KKT, and plateau diagnostics used to select the output policy\;
}
Construct a continuous average $\bar a_{\rm raw}$ of post-burn-in pre-projection control fields\;
\KwOut{$u_{\rm return}=\Pi_{\U}(\InterpOp[\bar a_{\rm raw}])$.}
\end{algorithm}

The CRRA specialization in \cref{alg:ipgdpo} represents only $R_{\mathrm{OL}}$ and solves the resulting strongly concave portfolio QP up to numerical tolerance.

\paragraph{Rollout and output policies.}
The rollout $u^k$ is damped because it determines the next state distribution and adjoint estimate. The output policy is formed from the averaged pre-projection control field. With burn-in index $k_0$ and weights $\omega_k\ge0$ satisfying $\sum_{k=k_0}^{N_{\mathrm{out}}-1}\omega_k=1$,
\begin{equation}
  \bar a_{\mathrm{raw}}(t,s)=\sum_{k=k_0}^{N_{\mathrm{out}}-1}\omega_k a_{\mathrm{raw}}^k(t,s),
  \qquad
  u_{\mathrm{return}}(t,s)=\Pi_{\U(t,s)}\!\left(\InterpOp[\bar a_{\mathrm{raw}}](t,s)\right).
  \label{eq:structure-return}
\end{equation}
Interpolation or regression is applied before the final projection. Projecting at nodes and then interpolating can move a switching boundary; averaging projected rollout policies can blur it because damping leaves finite-iteration values strictly inside active bounds.

\paragraph{Coupled time-grid correction.}
When first-order time-discretization bias dominates conditional-estimation noise, the paired-grid estimator
\begin{equation}
  \widehat\vartheta_{\mathrm{RE}}^u
  =2\widehat\vartheta_{h/2}^u-\InterpOp[\widehat\vartheta_h^u]
  \label{eq:richardson}
\end{equation}
removes the leading $O(h)$ term under the usual coupled expansion \citep{Richardson1911,TalayTubaro1990}. The experiments call this coupled Richardson correction; $\InterpOp$ transfers the coarse-grid field to the fine state--time design.

The matched-budget comparison tests whether current-policy re-evaluation and pooled initial-policy refinement target different policy-improvement operators; it is not an assumption of the convergence results.

\section{Adjoint--HJB consistency and value improvement}
\label{sec:theory}

The analysis has two steps. First, we compare the population fixed-latent adjoint Hamiltonian with the fixed-policy HJB Hamiltonian and show that their gradient discrepancy at the updated action is scaled by the policy-improvement residual. Second, a performance-difference identity transfers this discrepancy, together with finite numerical errors, to a true-value improvement bound. Technical PDE and BSDE estimates are collected in Appendix~\ref{app:supp-generic}.
\subsection{HJB reference and residual bounds for fixed-latent OL-BPTT}
\label{sec:compatibility-theory}

For a fixed feedback $u$, write
\begin{equation}
  b^u(t,s)=b(t,s,u(t,s)),
  \qquad
  \sigma^u(t,s)=\sigma(t,s,u(t,s)),
  \qquad
  \ell^u(t,s)=\ell(t,s,u(t,s)).
\end{equation}
Throughout this subsection we work in direct action coordinates, equivalently under a state-independent feasibility chart. This is the setting of the iterative benchmarks and makes the deployed action itself the fixed-latent selector. For a moving state-dependent chart, fixed-latent differentiation retains the structural chart derivative as described in \cref{sec:adjoint-recovery}; the chart-composed residual identity is not asserted here. In direct coordinates, fixed-latent OL-BPTT holds the deployed feedback action fixed when differentiating a rollout. To quantify its target error, we introduce the fully closed-loop state tangent solely as the analytical device that identifies the exact fixed-policy value derivative. It is not used as the numerical acquisition rule in the reported algorithms. Write $u_r:=u(r,S_r^u)$ and $I=I_{d_S}$. The closed-loop reference tangent $J_r^u:=\partial S_r^u/\partial s$ solves
\begin{equation}
\begin{split}
  \dd J_r^u={}&\bigl[b_s+b_a\D_su\bigr](r,S_r^u,u_r)J_r^u\dd r\\
  &+\sum_{j=1}^{d_W}\bigl[\sigma_s^{(j)}+\sigma_a^{(j)}\D_su\bigr](r,S_r^u,u_r)J_r^u\dd W_r^{(j)},
  \qquad J_t^u=I.
  \label{eq:closed-loop-tangent}
\end{split}
\end{equation}

\begin{lemma}[Fixed-policy value-gradient identification and exact fixed-latent defect]
\label{lem:closed-loop-compatibility}
Under the direct-action convention above, let $\mathcal O\subseteq\R^{d_S}$ be an invariant open state domain. Fix a restart point $(t,s)$ and an open future region $\mathcal D\subset[0,T]\times\mathcal O$ containing the $u$-trajectories generated by all sufficiently small perturbations of $s$. Assume that the admissible feedback $u$ is $C^1$ on a neighborhood of $\mathcal D$ with bounded state derivative there. Suppose $b$, $\sigma$, and $\ell$ are continuously differentiable in $(s,a)$ with bounded first derivatives on the corresponding state--action neighborhood and $g\in C^1(\mathcal O)$. Assume that $g(S_T^u)$ and $\int_t^T|\ell^u(r,S_r^u)|\dd r$ are integrable, that the terminal- and running-payoff difference quotients generated by sufficiently small perturbations of the initial state are uniformly integrable (the running family under $\dd r\otimes\dd\Pp$), and that the state and tangent processes have the conjugate moments needed below. Then $V^u(t,\cdot)$ is differentiable at $s$ (and throughout any restart neighborhood satisfying the same localized conditions), and
\begin{equation}
\begin{split}
  \D_sV^u(t,s)
  =\E_{t,s}^u\Bigg[
  &(J_T^u)^\top\D g(S_T^u)\\
  &+\int_t^T (J_r^u)^\top
  \Bigl\{\ell_s+ (\D_su)^\top\ell_a\Bigr\}
  (r,S_r^u,u(r,S_r^u))\dd r
  \Bigg].
  \label{eq:closed-loop-value-gradient}
\end{split}
\end{equation}
Equation~\eqref{eq:closed-loop-value-gradient} is the exact fixed-policy value-gradient reference.

Along the same fixed-policy path, let $J_{\mathrm{FL}}^u$ be the fixed-latent forward tangent obtained by dropping $b_a\D_su$ and $\sigma_a\D_su$. It is the forward sensitivity dual to fixed-latent OL-BPTT, in which the state dependence of the deployed feedback is not included in the reverse-mode chain rule. Define the corresponding population first-order target by
\begin{equation}
  G_{\mathrm{OL}}^u(t,s)
  =\E_{t,s}^u\left[
  (J_{\mathrm{FL},T}^u)^\top\D g(S_T^u)
  +\int_t^T(J_{\mathrm{FL},r}^u)^\top\ell_s(r,S_r^u,u_r)\dd r
  \right].
  \label{eq:fixed-latent-target}
\end{equation}
Then the exact fixed-latent OL-BPTT defect is
\begin{equation}
\begin{split}
  \D_sV^u-G_{\mathrm{OL}}^u
  =\E_{t,s}^u\Bigg[
  &(J_T^u-J_{\mathrm{FL},T}^u)^\top\D g(S_T^u)\\
  &+\int_t^T\Bigl[
  (J_r^u-J_{\mathrm{FL},r}^u)^\top\ell_s
  +(J_r^u)^\top(\D_su)^\top\ell_a
  \Bigr]\dd r
  \Bigg].
  \label{eq:fixed-latent-defect}
\end{split}
\end{equation}
Hence a decoder supplied with the exact fixed-policy value derivative in \cref{eq:closed-loop-value-gradient} has zero first-order consistency error when its remaining Hamiltonian blocks are exact, whereas the computational fixed-latent OL-BPTT route has the explicit, auditable finite-iteration defect in \cref{eq:fixed-latent-defect}.
\end{lemma}

\begin{proof}
After a bounded $C^1$ extension outside the relevant future region, the stochastic flow generated by the closed-loop coefficients is differentiable in its initial state, and its derivative solves \cref{eq:closed-loop-tangent}; see Kunita~\citep{Kunita1990}. The resulting identities are independent of the chosen extension because the localized trajectories remain in $\mathcal D$. The uniform-integrability envelope, stochastic-flow convergence, and Vitali's theorem justify differentiation of the fixed-policy Feynman--Kac representation under the expectation and give \cref{eq:closed-loop-value-gradient}. Subtracting the fixed-latent OL-BPTT target gives \cref{eq:fixed-latent-defect}. Details are given in Appendix~\ref{app:compatibility-proof}.
\end{proof}

\begin{corollary}[Residual bound for the first-order fixed-latent OL-BPTT defect]
\label{cor:olbptt-residual-defect}
Assume the conditions of \cref{lem:closed-loop-compatibility}, with $\mathcal D$ denoting the corresponding future region. Suppose additionally that the fixed-policy equation is classically differentiable in the state and that the resulting vector first-variation equations admit the Feynman--Kac representations generated by $J_{\mathrm{FL}}^u$; for example, it is sufficient locally that $V^u\in C^{1,3}$ and that the derivatives appearing below are continuous and bounded. Define the envelope source
\begin{equation}
  r_{\mathrm{env}}^u(t,s)
  :=(\D_su(t,s))^\top
  \nabla_a\Qcal^u\bigl(t,s,u(t,s)\bigr).
  \label{eq:olbptt-envelope-source}
\end{equation}
Then the first-order fixed-latent OL-BPTT defect has the exact representation
\begin{equation}
  \D_sV^u(t,s)-G_{\mathrm{OL}}^u(t,s)
  =\E_{t,s}^u\left[
    \int_t^T
    (J_{\mathrm{FL},r}^u)^\top
    r_{\mathrm{env}}^u(r,S_r^u)\dd r
  \right].
  \label{eq:olbptt-envelope-representation}
\end{equation}
Consequently, with
\begin{equation}
  C_{\mathrm{flow}}^u(t,s)
  :=\E_{t,s}^u\left[
    \int_t^T\norm{J_{\mathrm{FL},r}^u}_{\mathrm{op}}\dd r
  \right],
  \label{eq:olbptt-flow-constant}
\end{equation}
one has
\begin{equation}
  \norm{\D_sV^u(t,s)-G_{\mathrm{OL}}^u(t,s)}
  \le
  C_{\mathrm{flow}}^u(t,s)
  \norm{r_{\mathrm{env}}^u}_{\infty;\mathcal D}.
  \label{eq:olbptt-stationarity-bound}
\end{equation}

Suppose further that, on $\mathcal D$, the difference of the fixed-policy HJB action Hamiltonian factorizes as
\begin{equation}
  \Qcal^u(t,s,a)-\Qcal^u(t,s,u(t,s))
  =w^u(t,s)\bigl[q^u(t,s,a)-q^u(t,s,u(t,s))\bigr],
  \qquad w^u>0,
  \label{eq:weighted-quadratic-advantage}
\end{equation}
where
\begin{equation}
  q^u(t,s,a)=c_0^u(t,s)+(m^u(t,s))^\top a
  -\frac12a^\top H^u(t,s)a,
  \qquad H^u(t,s)\succ0.
  \label{eq:quadratic-score-general}
\end{equation}
Assume pointwise on $\mathcal D$ either that $\Tcal(u)$ is interior, or that $u$ and $\Tcal(u)$ lie in the relative interior of the same regular active face and the columns of $\D_su$ are tangent to that face. Then
\begin{equation}
  r_{\mathrm{env}}^u
  =w^u(\D_su)^\top H^u\bigl(\Tcal(u)-u\bigr)
  =w^u(\D_su)^\top H^u\ResOp(u).
  \label{eq:olbptt-envelope-residual-identity}
\end{equation}
Define
\begin{equation}
  \norm{\ResOp(u)}_{\infty,H;\mathcal D}
  :=\sup_{(r,x)\in\mathcal D}
  \norm{\Tcal(u)(r,x)-u(r,x)}_{H^u(r,x)}
  \label{eq:olbptt-weighted-residual-norm}
\end{equation}
and
\begin{equation}
  C_{\mathrm{OL}}^u(t,s)
  :=\E_{t,s}^u\left[
    \int_t^T
    \norm{J_{\mathrm{FL},r}^u}_{\mathrm{op}}
    w^u(r,S_r^u)
    \norm{(\D_su(r,S_r^u))^\top(H^u(r,S_r^u))^{1/2}}_{\mathrm{op}}
    \dd r
  \right].
  \label{eq:olbptt-residual-constant}
\end{equation}
Whenever this quantity is finite,
\begin{equation}
  \norm{\D_sV^u(t,s)-G_{\mathrm{OL}}^u(t,s)}
  \le
  C_{\mathrm{OL}}^u(t,s)
  \norm{\ResOp(u)}_{\infty,H;\mathcal D}.
  \label{eq:olbptt-residual-scaled-bound}
\end{equation}
In particular, on such a region, a fixed point of the HJB policy-improvement operator $\Tcal(u)=u$ has zero first-order fixed-latent OL-BPTT target defect. If, along a policy sequence on a common future region $\mathcal D$, $C_{\mathrm{OL}}^u$ is uniformly bounded on compact sets of restart points and $\norm{\ResOp(u)}_{\infty,H;\mathcal D}\to0$, then the first-order OL-BPTT defect converges locally uniformly to zero.
\end{corollary}

\begin{proof}
Under the additional classical regularity, differentiating the fixed-policy equation shows that $\D_sV^u$ and $G_{\mathrm{OL}}^u$ satisfy the same vector linear backward operator generated by the fixed-latent tangent, with their difference driven only by the source in \cref{eq:olbptt-envelope-source}. The vector Feynman--Kac formula gives \cref{eq:olbptt-envelope-representation}, and \cref{eq:olbptt-stationarity-bound} follows immediately. For the quadratic Hamiltonian,
\[
  \nabla_a q^u(u)=\nabla_a q^u(\Tcal(u))
  +H^u\bigl(\Tcal(u)-u\bigr).
\]
The first term is zero at an interior maximizer. On a common regular active face it is normal to the face, while the columns of $\D_su$ are tangent, so it vanishes after left multiplication by $(\D_su)^\top$. \Cref{eq:olbptt-envelope-residual-identity,eq:olbptt-residual-scaled-bound} then follow from the operator-norm inequality. A coordinate derivation of the vector equation is given in Appendix~\ref{app:compatibility-proof}.
\end{proof}

\begin{remark}[Role of the first-order estimate]
The first-adjoint defect is proportional to the policy-improvement residual, but this alone does not control the Hamiltonian-gradient discrepancy used by the value certificate at the approximate maximizer. Proposition~\ref{prop:shifted-gradient-cancellation} confines the second-adjoint discrepancy to the trial-action diffusion displacement; an interior parabolic estimate controls the remaining martingale block. These local estimates do not by themselves establish the gain-dependent directional bound or asymptotic stationarity compatibility required for global convergence.
\end{remark}
\paragraph{Extension to the complete adjoint--HJB Hamiltonian gradient.}
\label{sec:decoded-score-compatibility}

Write
\begin{equation}
  p^u:=\D_sV^u,
  \qquad
  \Gamma^u:=\D^2_{ss}V^u,
  \qquad
  \sigma^u(t,s):=\sigma(t,s,u(t,s)),
  \label{eq:generic-decoded-reference-fields}
\end{equation}
and let
\begin{equation}
  \Hgen_{\mathrm A}^u(t,s,a)
  :=\Hgen\bigl(t,s;a,\lambda^u(t,s),\zeta^u(t,s),P^u(t,s)\bigr)
  \label{eq:population-adjoint-hamiltonian}
\end{equation}
be the population fixed-latent generalized Hamiltonian.  For a matrix-valued action map $a\mapsto\sigma(t,s,a)$, define
\begin{equation}
  \bigl[\D_a\sigma(t,s,a)\bigr]^\ast[\Xi]
  :=\left(
    \inner{\partial_{a_j}\sigma(t,s,a)}{\Xi}_F
  \right)_{j=1}^{d_u}.
  \label{eq:action-diffusion-adjoint}
\end{equation}

\begin{proposition}[Exact shifted-adjoint cancellation]
\label{prop:shifted-gradient-cancellation}
Suppose $P^u$ and $\Gamma^u$ are symmetric.  Then, at every trial action $a$ for which the displayed derivatives exist,
\begin{align}
  &\nabla_a\Hgen_{\mathrm A}^u(t,s,a)-\nabla_a\Qcal^u(t,s,a)
  \notag\\
  &\quad={}
  b_a(t,s,a)^\top\bigl(\lambda^u-p^u\bigr)
  \label{eq:shifted-gradient-cancellation}\\
  &\qquad
  +\bigl[\D_a\sigma(t,s,a)\bigr]^\ast
  \left[
    \mathsf Z^u-\Gamma^u\sigma^u
    +\bigl(P^u-\Gamma^u\bigr)
      \bigl(\sigma(t,s,a)-\sigma^u\bigr)
  \right].
  \notag
\end{align}
In particular, at $a=u(t,s)$ the second-adjoint discrepancy cancels completely:
\begin{equation}
  \nabla_a\Hgen_{\mathrm A}^u(t,s,u)
  -\nabla_a\Qcal^u(t,s,u)
  =b_a(t,s,u)^\top(\lambda^u-p^u)
   +\bigl[\D_a\sigma(t,s,u)\bigr]^\ast
    \bigl[\mathsf Z^u-\Gamma^u\sigma^u\bigr].
  \label{eq:shifted-gradient-current-action}
\end{equation}
\end{proposition}

\begin{proof}
Differentiate \cref{eq:generalized-hamiltonian}.  Symmetry of $P^u$ gives
\[
  \nabla_a\Hgen_{\mathrm A}^u(a)
  =\ell_a+b_a^\top\lambda^u
   +[\D_a\sigma(a)]^\ast\bigl[\zeta^u+P^u\sigma(a)\bigr].
\]
Since $\zeta^u=\mathsf Z^u-P^u\sigma^u$,
$\zeta^u+P^u\sigma(a)=\mathsf Z^u+P^u(\sigma(a)-\sigma^u)$.
Differentiating \cref{eq:markov-score} gives
$\nabla_a\Qcal^u(a)=\ell_a+b_a^\top p^u+[\D_a\sigma(a)]^\ast[\Gamma^u\sigma(a)]$.
Subtracting and adding $\Gamma^u\sigma^u$ yields \cref{eq:shifted-gradient-cancellation}; setting $a=u$ yields \cref{eq:shifted-gradient-current-action}.
\end{proof}

\paragraph{Regular future tube.}
Fix compact parabolic cylinders
\[
  K\Subset K_1\Subset[0,T-\tau]\times\mathcal O,\qquad \tau>0.
\]
An open set $\mathcal D\subset[0,T)\times\mathcal O$ is a regular future tube for $(K_1,u)$ if, for every restart $(t,s)\in K_1$, the corresponding $u$-trajectory remains in $\mathcal D$ on $[t,T)$ almost surely, and throughout $\mathcal D$ either $\Tcal(u)$ is interior or $u$ and $\Tcal(u)$ lie in the relative interior of the same regular active face, with the columns of $\D_su$ tangent to that face. Set
\begin{equation}
  \mathfrak r_{\mathcal D}(u)
  :=\norm{\Tcal(u)-u}_{L^\infty(\mathcal D)}.
  \label{eq:future-tube-residual}
\end{equation}

\paragraph{Scope.}
The regular-tube result is intentionally a fixed-stratum local statement. Under nondegenerate state noise, requiring every restarted path to remain in $\mathcal D$ is a strong reachable-set condition and can effectively require one regular face on the region reached from $K_1$; the result is therefore not intended to cover trajectories crossing active-set switching regions. Such regions are not controlled by Theorem~\ref{thm:generic-decoded-score-compatibility}. The global CRRA analysis instead uses the directional occupation and asymptotic stationarity-compatibility hypotheses in \cref{ass:crra-adjoint-compatible}; the audited norm-relative condition is a stronger sufficient route. No pointwise residual-scaled estimate across a switching surface is asserted.

\begin{assumption}[Regularity for adjoint--HJB consistency]
\label{ass:generic-decoded-score-region}
On a regular future tube $\mathcal D$, the direct-action coefficients have bounded derivatives through order two, $V^u\in C^{1,3}$, and the fixed-latent first-adjoint defect admits the local parabolic estimate implied by \cref{cor:olbptt-residual-defect}. The restarted second-adjoint BSDE has a locally bounded solution, and on $K$ the policies $u$, $\Tcal(u)$, and $\Acal(u)$ lie in the relative interior of one common regular affine face. Along that face, the adjoint and HJB Hamiltonians are uniformly strongly concave; $b_a$, $\D_a\sigma$, and the HJB action curvature are bounded, and $\sigma$ is locally Lipschitz in the action. Whenever normalized Hamiltonians are used, their common positive action-independent normalizer and its reciprocal are locally bounded on $K$. The precise moment and interior-estimate conditions sufficient for these statements are given in Appendix~\ref{app:supp-generic}.
\end{assumption}

\begin{theorem}[Residual bound for the adjoint--HJB Hamiltonian-gradient discrepancy]
\label{thm:generic-decoded-score-compatibility}
Under \cref{ass:generic-decoded-score-region}, there is a finite local constant $C_{\mathrm{dec}}$ such that
\begin{equation}
  \sup_{(t,s)\in K}
  \norm{
    \nabla_a\Hgen_{\mathrm A}^u\bigl(t,s,\Acal(u)(t,s)\bigr)
    -\nabla_a\Qcal^u\bigl(t,s,\Acal(u)(t,s)\bigr)
  }
  \le C_{\mathrm{dec}}\,\mathfrak r_{\mathcal D}(u),
  \label{eq:generic-decoded-score-residual}
\end{equation}
and, if the HJB action Hamiltonian is $\mu_{\mathrm M}$-strongly concave on the common face,
\begin{equation}
  \norm{\Acal(u)-\Tcal(u)}_{L^\infty(K)}
  \le \frac{C_{\mathrm{dec}}}{\mu_{\mathrm M}}
  \mathfrak r_{\mathcal D}(u).
  \label{eq:generic-map-distance-residual}
\end{equation}
Thus every regular fixed point of $\Tcal$ is also a fixed point of $\Acal$, even though $P^u$ need not approach $\D^2_{ss}V^u$.
\end{theorem}

\begin{proof}
The first-adjoint value-gradient and martingale-coefficient defects are $O(\mathfrak r_{\mathcal D}(u))$ by the fixed-latent envelope equation and an interior parabolic estimate. The second adjoint is locally bounded. By \cref{prop:shifted-gradient-cancellation}, its discrepancy from the value Hessian is multiplied by $\sigma(\Acal(u))-\sigma(u)$; strong concavity and the first-adjoint estimate make this policy displacement $O(\mathfrak r_{\mathcal D}(u))$. Combining the terms gives \cref{eq:generic-decoded-score-residual}, and the variational inequalities for the two strongly concave maximization problems give \cref{eq:generic-map-distance-residual}. Appendix~\ref{app:decoded-score-compatibility-proof} gives the complete estimates.
\end{proof}

\subsection{Quadratic policy improvement and value-improvement bounds}
\label{sec:value-improvement}

For a trial action $a$, define the controlled generator
\begin{equation}
  \Lcal^a\phi(t,s)
  =b(t,s,a)^\top\D_s\phi(t,s)
  +\frac12\tr\!\left[\sigma\sigma^\top(t,s,a)\D^2_{ss}\phi(t,s)\right],
  \qquad
  \Lcal^u\phi:=\Lcal^{u(t,s)}\phi.
  \label{eq:controlled-generator}
\end{equation}
The fixed-policy action advantage of replacing $u(t,s)$ by $a$ is
\begin{equation}
  \Adv^u(t,s,a)
  :=\ell(t,s,a)+\Lcal^aV^u(t,s)
  -\ell^u(t,s)-\Lcal^uV^u(t,s)
  =\Qcal^u(t,s,a)-\Qcal^u(t,s,u(t,s)).
  \label{eq:markov-advantage}
\end{equation}
For the value-improvement results, retain the factorization and normalized quadratic Hamiltonian in \cref{eq:weighted-quadratic-advantage,eq:quadratic-score-general}, take $q^u=q_{\mathrm{HJB}}^u$, and strengthen the curvature bound to $0<\underline h I\preceq H^u\preceq\overline h I$. The linear and quadratic coefficients may be obtained from first/second adjoints or supplied partly by model structure. The pointwise feasible set $\Kset:=\U(t,s)$ is closed and convex.

\begin{assumption}[Performance-difference regularity]
\label{ass:performance-difference}
For each reference feedback $u$ and alternative feedback $v$ used below, $V^u\in C^{1,2}$ solves the fixed-policy equation
\begin{equation}
  \partial_tV^u+\ell^u+\Lcal^uV^u=0,\qquad V^u(T,\cdot)=g,
  \label{eq:fixed-policy-backward-equation}
\end{equation}
on an invariant state domain. Localization and uniform-integrability conditions justify It\^o's formula under the law of $v$ and passage to the terminal time. Consequently,
\begin{equation}
  V^v(t,s)-V^u(t,s)
  =\E_{t,s}^v\left[
    \int_t^T\Adv^u(r,S_r^v,v(r,S_r^v))\dd r
  \right].
  \label{eq:performance-difference}
\end{equation}
For the CRRA models below, the stated affine-factor and bounded-policy assumptions, together with the corresponding negative wealth moments, provide this integrability.
\end{assumption}

For a closed convex set $\Kset$, use
\begin{equation}
  N_{\Kset}(a)
  :=\left\{v:\ \inner{v}{b-a}\le0\ \text{for every }b\in\Kset\right\}.
  \label{eq:normal-cone}
\end{equation}
\begin{lemma}[Stability of the greedy control under Hamiltonian-gradient error]
\label{lem:quadratic-regret}
At a fixed $(t,s)$, let $a^\star$ maximize the quadratic HJB Hamiltonian $q^u$ over $\Kset$, and let $\widehat a$ maximize a differentiable represented Hamiltonian $\widehat q^u$ over the same set. Define
\begin{equation}
  \varepsilon_{\mathrm{grad}}
  :=\norm{\nabla_a\widehat q^u(\widehat a)-\nabla_a q^u(\widehat a)}_{(H^u)^{-1}}.
\end{equation}
Then
\begin{equation}
  \norm{\widehat a-a^\star}_{H^u}\le\varepsilon_{\mathrm{grad}},
  \qquad
  0\le q^u(a^\star)-q^u(\widehat a)
  \le\frac12\varepsilon_{\mathrm{grad}}^2.
  \label{eq:quadratic-regret-bound}
\end{equation}
Thus errors in either the linear adjoint coefficient or the approximated curvature are controlled through the gradient mismatch at the approximate maximizer.
\end{lemma}

\begin{proof}
The optimality conditions are
$\nabla q^u(a^\star)\in N_{\Kset}(a^\star)$ and
$\nabla\widehat q^u(\widehat a)\in N_{\Kset}(\widehat a)$.
Monotonicity of the normal cone gives
\[
\begin{split}
  0&\le
  \inner{\nabla q^u(a^\star)-\nabla\widehat q^u(\widehat a)}
  {a^\star-\widehat a}\\
  &=-\norm{\widehat a-a^\star}_{H^u}^2
  +\inner{\nabla\widehat q^u(\widehat a)-\nabla q^u(\widehat a)}
  {\widehat a-a^\star},
\end{split}
\]
which proves the distance bound. Exact quadraticity gives
\[
  q^u(a^\star)-q^u(\widehat a)
  =\nabla q^u(\widehat a)^\top(a^\star-\widehat a)
  -\frac12\norm{\widehat a-a^\star}_{H^u}^2.
\]
The variational inequality for $\widehat q^u$ bounds the first term by
$\varepsilon_{\mathrm{grad}}\norm{\widehat a-a^\star}_{H^u}$; completing the square yields the regret bound.
\end{proof}

\begin{proposition}[Exact damped value improvement]
\label{thm:exact-value-improvement}
Let
\begin{equation}
  u_\beta=(1-\beta)u+\beta\Tcal(u),
  \qquad 0<\beta\le1.
\end{equation}
Suppose \cref{ass:performance-difference} holds for $(u,u_\beta)$, \cref{eq:weighted-quadratic-advantage} holds, and $\Kset$ is convex. Then $u_\beta$ is admissible and
\begin{equation}
\begin{split}
  V^{u_\beta}(t,s)-V^u(t,s)
  \ge{}&\frac{\beta(2-\beta)}{2}
  \E_{t,s}^{u_\beta}\left[
  \int_t^T w^u(r,S_r^{u_\beta})
  \norm{\Tcal(u)-u}_{H^u}^2(r,S_r^{u_\beta})\dd r
  \right].
  \label{eq:exact-value-improvement}
\end{split}
\end{equation}
Thus damping is a conservative action interpolation with an explicit, nonnegative improvement certificate. This is analogous in purpose to conservative policy-iteration analyses \citep{KakadeLangford2002}, but here the mixture is deterministic in a convex continuous action set and the lower bound follows from the explicit quadratic policy-improvement geometry.
\end{proposition}

\begin{proof}
Pointwise, let $d=\Tcal(u)-u$. The variational inequality at the constrained maximizer gives $\nabla q^u(\Tcal(u))^\top d\ge0$. Since $q^u$ is exactly quadratic,
\[
  q^u(u+\beta d)-q^u(u)
  =\beta\nabla q^u(u)^\top d
  -\frac{\beta^2}{2}\norm{d}_{H^u}^2
  \ge\frac{\beta(2-\beta)}2\norm{d}_{H^u}^2.
\]
Insert this inequality into \cref{eq:performance-difference} under the new policy law.
\end{proof}

\paragraph{Three-level operator hierarchy.}
For the finite-computation statement, suppress the outer index and write $\AcalNum:=\AcalNum_k$. The hierarchy
\[
  \AcalNum_k \longrightarrow \Acal \longrightarrow \Tcal
\]
separates numerical error from population-level adjoint--HJB consistency. Corollary~\ref{cor:relative-adjoint-step} isolates the second arrow by setting $\AcalNum=\Acal$ and removing discretization, sampling, representation, continuation-extension, and QP-solver errors.

The finite-computation theorem uses the gradient error of the represented action Hamiltonian rather than a single end-to-end policy-update norm. Let $\widehat q^u$ be the normalized quadratic Hamiltonian constructed from the estimated adjoint field and the QP decoder. At a feasible action $a$, write
\begin{equation}
  e_{\mathrm{ham}}^u(a)
  :=\nabla_a\widehat q^u(a)-\nabla_a q^u(a)
  =e_{\mathrm{cons}}^u(a)+e_{\mathrm{disc}}^u(a)+e_{\mathrm{stat}}^u(a)
  +e_{\mathrm{repr}}^u(a)+e_{\mathrm{ext}}^u(a).
  \label{eq:score-error-decomposition}
\end{equation}
The five components are the adjoint--HJB discrepancy, time discretization, statistical/conditional estimation, representation, and continuation extension. They are measured after propagation through the adjoint-to-Hamiltonian construction. The continuation-extension term collects terminal, boundary/tail, and off-grid rules needed to complete a locally represented Hamiltonian. Any component may affect both the linear coefficient and the quadratic curvature. For the computed OL-BPTT policy update $\AcalNum(u)$, define
\begin{equation}
  \varepsilon_j(u)
  :=\norm{e_j^u(\AcalNum(u))}_{(H^u)^{-1}},
  \qquad
  j\in\{\mathrm{cons},\mathrm{disc},\mathrm{stat},\mathrm{repr},\mathrm{ext}\},
\end{equation}
and the pointwise Hamiltonian-gradient error bound
\begin{equation}
  \epsham(u)
  :=\norm{e_{\mathrm{ham}}^u(\AcalNum(u))}_{(H^u)^{-1}}
  \le\epscons(u)+\epsdisc(u)+\epsstat(u)+\epsrepr(u)+\epsext(u).
  \label{eq:score-gradient-error-radius}
\end{equation}
For a numerically computed policy update in a bounded constrained action set, define the computable KKT gap of $\widehat q^u$ by \citep{NocedalWright2006}
\begin{equation}
  \epsrec(u)
  :=\left[
  \sup_{a\in\Kset}
  \inner{\nabla_a\widehat q^u(\AcalNum(u))}
  {a-\AcalNum(u)}
  \right]_+.
  \label{eq:recovery-kkt-gap}
\end{equation}
Here $[x]_+:=\max\{x,0\}$. The numerical sources are therefore five Hamiltonian-gradient error components plus the local optimality gap. For an exact quadratic-program solve, $\epsrec=0$.

\begin{theorem}[Approximate value improvement under decomposed numerical errors]
\label{thm:approx-value-improvement}
Set
\begin{equation}
  d^u=\AcalNum(u)-u,
  \qquad
  u^+=u+\beta d^u,
  \qquad 0<\beta\le1.
\end{equation}
Suppose \cref{ass:performance-difference} holds for $(u,u^+)$, \cref{eq:weighted-quadratic-advantage} holds, and $\Kset$ is convex. Conditional on the represented adjoint Hamiltonian, the pointwise HJB Hamiltonian gain satisfies
\begin{align}
  q^u(u+\beta d^u)-q^u(u)
  &\ge\beta\left[
  \frac{2-\beta}{2}\norm{d^u}_{H^u}^2
  -\inner{e_{\mathrm{ham}}^u(\AcalNum(u))}{d^u}
  -\epsrec(u)
  \right]
  \label{eq:approx-score-gain-directional}\\
  &\ge\beta\left[
  \frac{2-\beta}{2}\norm{d^u}_{H^u}^2
  -\epsham(u)\norm{d^u}_{H^u}
  -\epsrec(u)
  \right]
  \label{eq:approx-score-gain-sharp}\\
  &\ge\beta\left[
  \frac{2-\beta}{4}\norm{d^u}_{H^u}^2
  -\frac{\epsham(u)^2}{2-\beta}
  -\epsrec(u)
  \right].
  \label{eq:approx-score-gain-young}
\end{align}
Consequently,
\begin{equation}
\begin{split}
  V^{u^+}(t,s)-V^u(t,s)
  \ge\beta\E_{t,s}^{u^+}\Bigg[
  \int_t^T w^u
  \left[
  \frac{2-\beta}{4}\norm{d^u}_{H^u}^2
  -\frac{\epsham(u)^2}{2-\beta}
  -\epsrec(u)
  \right]\dd r
  \Bigg],
  \label{eq:approx-value-improvement}
\end{split}
\end{equation}
where all integrands are evaluated along $S^{u^+}$. A computed adjoint-based policy update is therefore certified whenever its weighted policy-update magnitude dominates the squared Hamiltonian-gradient error and the local optimality gap.
\end{theorem}

\begin{proof}
Let $d=d^u$ and $\widetilde a=\AcalNum(u)$. By the KKT-gap definition,
$\nabla\widehat q^u(\widetilde a)^\top d\ge-\epsrec(u)$.
Exact quadraticity of the HJB action Hamiltonian gives
\[
  \nabla q^u(u)
  =\nabla q^u(\widetilde a)+H^ud
  =\nabla\widehat q^u(\widetilde a)-e_{\mathrm{ham}}^u(\widetilde a)+H^ud.
\]
Combining this identity with the KKT-gap inequality gives
\cref{eq:approx-score-gain-directional}. Cauchy--Schwarz gives
\cref{eq:approx-score-gain-sharp}; Young's inequality gives
\cref{eq:approx-score-gain-young}. The value statement follows from
\cref{eq:performance-difference}.
\end{proof}
\begin{corollary}[Pointwise relative-error transfer for the adjoint-based operator]
\label{cor:relative-adjoint-step}
At a fixed $(t,s)$, let $\Acal(u)$ be the exact maximizer over $\Kset$ of a differentiable population fixed-latent adjoint Hamiltonian $q_{\mathrm A}^u$, and let $\Tcal(u)$ maximize the quadratic HJB action Hamiltonian $q_{\mathrm{HJB}}^u$ with curvature $H^u\succ0$. Set
\begin{equation}
  d_{\mathrm A}^u:=\Acal(u)-u,
  \qquad
  c_\beta:=\frac{2-\beta}{2},
  \qquad 0<\beta\le1,
  \label{eq:relative-adjoint-definitions}
\end{equation}
and define the adjoint--HJB Hamiltonian-gradient discrepancy at the adjoint-based policy update by
\begin{equation}
  e_{\mathrm A}^u
  :=\nabla_a q_{\mathrm A}^u(\Acal(u))
  -\nabla_a q_{\mathrm{HJB}}^u(\Acal(u)),
  \qquad
  \varepsilon_{\mathrm A}(u):=\norm{e_{\mathrm A}^u}_{(H^u)^{-1}}.
  \label{eq:population-adjoint-score-error}
\end{equation}
Suppose that, pointwise,
\begin{equation}
  \varepsilon_{\mathrm A}(u)
  \le \kappa\norm{d_{\mathrm A}^u}_{H^u}
  \qquad\text{for some }0\le\kappa<c_\beta.
  \label{eq:relative-population-score-condition}
\end{equation}
Then the damped adjoint-based policy update $u_{\mathrm A}^+=u+\beta d_{\mathrm A}^u$ satisfies
\begin{equation}
  q_{\mathrm{HJB}}^u(u_{\mathrm A}^+)-q_{\mathrm{HJB}}^u(u)
  \ge
  \beta(c_\beta-\kappa)
  \norm{\Acal(u)-u}_{H^u}^2.
  \label{eq:relative-adjoint-score-improvement}
\end{equation}
Moreover,
\begin{equation}
  \norm{\Acal(u)-\Tcal(u)}_{H^u}
  \le \kappa\norm{\Acal(u)-u}_{H^u},
  \label{eq:adjoint-markov-map-distance}
\end{equation}
and, since $\kappa<1$,
\begin{equation}
  \frac{1}{1+\kappa}\norm{\Tcal(u)-u}_{H^u}
  \le \norm{\Acal(u)-u}_{H^u}
  \le \frac{1}{1-\kappa}\norm{\Tcal(u)-u}_{H^u}.
  \label{eq:adjoint-markov-residual-equivalence}
\end{equation}
If \cref{ass:performance-difference} holds for $(u,u_{\mathrm A}^+)$ and the condition holds pointwise along the new-policy law, then
\begin{equation}
\begin{split}
  V^{u_{\mathrm A}^+}(t,s)-V^u(t,s)
  \ge{}&
  \frac{\beta(c_\beta-\kappa)}{(1+\kappa)^2}
  \E_{t,s}^{u_{\mathrm A}^+}\left[
  \int_t^T w^u
  \norm{\Tcal(u)-u}_{H^u}^2\dd r
  \right].
  \label{eq:relative-adjoint-value-improvement}
\end{split}
\end{equation}
Consequently, if $0<\underline\beta\le\beta_k\le\overline\beta\le1$, a common $\kappa<c_{\overline\beta}:=(2-\overline\beta)/2$ satisfies \cref{eq:relative-population-score-condition} along the population OL-BPTT iteration, and the values are bounded above, then the weighted HJB policy-improvement residuals are summable.
\end{corollary}

\begin{proof}
Apply \cref{thm:approx-value-improvement} with $\AcalNum=\Acal$, exact local optimization $\epsrec=0$, and Hamiltonian-gradient error bound $\epsham=\varepsilon_{\mathrm A}$. This gives \cref{eq:relative-adjoint-score-improvement}. The greedy-control stability estimate in \cref{lem:quadratic-regret} gives
$\norm{\Acal(u)-\Tcal(u)}_{H^u}\le\varepsilon_{\mathrm A}(u)$, proving \cref{eq:adjoint-markov-map-distance}. The triangle inequality yields
\[
  \norm{\Tcal(u)-u}_{H^u}
  \le(1+\kappa)\norm{\Acal(u)-u}_{H^u}
\]
and
\[
  (1-\kappa)\norm{\Acal(u)-u}_{H^u}
  \le\norm{\Tcal(u)-u}_{H^u},
\]
which proves \cref{eq:adjoint-markov-residual-equivalence}. Insert the first residual inequality into the performance-difference form of \cref{eq:relative-adjoint-score-improvement} to obtain \cref{eq:relative-adjoint-value-improvement}. Summing over $k$ proves the final statement because
$\beta_k(c_{\beta_k}-\kappa)\ge\underline\beta(c_{\overline\beta}-\kappa)>0$.
\end{proof}
\begin{corollary}[Occupation-measure directional and relative-error transfer]
\label{cor:occupancy-relative-adjoint-step}
Under the setup of \cref{cor:relative-adjoint-step}, fix a starting state $(t,s)$ and a gain $0<\beta\le1$, assume \cref{ass:performance-difference} for $(u,u_{\mathrm A,\beta}^+)$, and write $u_{\mathrm A,\beta}^+=u+\beta d_{\mathrm A}^u$. Along the law of $S^{u_{\mathrm A,\beta}^+}$, define
\begin{align}
  \bigl(D_{\mathrm A}^{u,\beta}(t,s)\bigr)^2
  &:={\E}_{t,s}^{u_{\mathrm A,\beta}^+}\left[
      \int_t^T w^u\norm{\Acal(u)-u}_{H^u}^2\dd r\right],
  \label{eq:occupancy-adjoint-movement}\\
  \bigl(E_{\mathrm A}^{u,\beta}(t,s)\bigr)^2
  &:={\E}_{t,s}^{u_{\mathrm A,\beta}^+}\left[
      \int_t^T w^u\varepsilon_{\mathrm A}(u)^2\dd r\right],
  \label{eq:occupancy-adjoint-defect}\\
  \bigl(R_{\mathrm{HJB}}^{u,\beta}(t,s)\bigr)^2
  &:={\E}_{t,s}^{u_{\mathrm A,\beta}^+}\left[
      \int_t^T w^u\norm{\Tcal(u)-u}_{H^u}^2\dd r\right],
  \label{eq:occupancy-markov-residual}\\
  \bigl(M_{\mathrm A}^{u,\beta}(t,s)\bigr)^2
  &:={\E}_{t,s}^{u_{\mathrm A,\beta}^+}\left[
      \int_t^T w^u\norm{\Acal(u)-\Tcal(u)}_{H^u}^2\dd r\right],
  \label{eq:occupancy-map-distance}\\
  \Xi_{\mathrm A}^{u,\beta}(t,s)
  &:={\E}_{t,s}^{u_{\mathrm A,\beta}^+}\left[
      \int_t^T w^u\inner{e_{\mathrm A}^u}{d_{\mathrm A}^u}\dd r\right].
  \label{eq:occupancy-directional-defect}
\end{align}
All integrands above are evaluated at $(r,S_r^{u_{\mathrm A,\beta}^+})$. Whenever these quantities are finite, the sharp quadratic-Hamiltonian estimate gives
\begin{equation}
  V^{u_{\mathrm A,\beta}^+}(t,s)-V^u(t,s)
  \ge \beta\left[
  c_\beta\bigl(D_{\mathrm A}^{u,\beta}(t,s)\bigr)^2
  -\Xi_{\mathrm A}^{u,\beta}(t,s)\right].
  \label{eq:occupancy-directional-base-bound}
\end{equation}
Consequently, the directional condition
\begin{equation}
  \Xi_{\mathrm A}^{u,\beta}(t,s)
  \le \kappa_{\mathrm{dir}}
  \bigl(D_{\mathrm A}^{u,\beta}(t,s)\bigr)^2,
  \qquad \kappa_{\mathrm{dir}}\in\R,\quad \kappa_{\mathrm{dir}}<c_\beta,
  \label{eq:occupancy-directional-condition}
\end{equation}
implies
\begin{equation}
  V^{u_{\mathrm A,\beta}^+}(t,s)-V^u(t,s)
  \ge \beta(c_\beta-\kappa_{\mathrm{dir}})
  \bigl(D_{\mathrm A}^{u,\beta}(t,s)\bigr)^2.
  \label{eq:occupancy-directional-value-improvement}
\end{equation}
If, more strongly,
\begin{equation}
  E_{\mathrm A}^{u,\beta}(t,s)
  \le \kappa D_{\mathrm A}^{u,\beta}(t,s),
  \qquad 0\le\kappa<c_\beta,
  \label{eq:occupancy-relative-population-score-condition}
\end{equation}
then \cref{eq:occupancy-directional-condition} holds with $\kappa_{\mathrm{dir}}=\kappa$, and
\begin{align}
  V^{u_{\mathrm A,\beta}^+}(t,s)-V^u(t,s)
  &\ge \beta(c_\beta-\kappa)
  \bigl(D_{\mathrm A}^{u,\beta}(t,s)\bigr)^2
  \label{eq:occupancy-relative-value-improvement}\\
  &\ge \frac{\beta(c_\beta-\kappa)}{(1+\kappa)^2}
  \bigl(R_{\mathrm{HJB}}^{u,\beta}(t,s)\bigr)^2.
  \label{eq:occupancy-relative-markov-improvement}
\end{align}
Moreover,
\begin{equation}
  M_{\mathrm A}^{u,\beta}(t,s)
  \le E_{\mathrm A}^{u,\beta}(t,s)
  \le \kappa D_{\mathrm A}^{u,\beta}(t,s),
  \label{eq:occupancy-map-distance-bound}
\end{equation}
and
\begin{equation}
  \frac{1}{1+\kappa}R_{\mathrm{HJB}}^{u,\beta}(t,s)
  \le D_{\mathrm A}^{u,\beta}(t,s)
  \le \frac{1}{1-\kappa}R_{\mathrm{HJB}}^{u,\beta}(t,s).
  \label{eq:occupancy-residual-equivalence}
\end{equation}
\end{corollary}

\begin{proof}
Applying \cref{eq:approx-score-gain-directional} with $\AcalNum=\Acal$ and $\epsrec=0$, and integrating along the updated-policy law, gives \cref{eq:occupancy-directional-base-bound}; \cref{eq:occupancy-directional-condition} then gives \cref{eq:occupancy-directional-value-improvement}. Cauchy--Schwarz in the weighted path--time space gives
$\Xi_{\mathrm A}^{u,\beta}\le E_{\mathrm A}^{u,\beta}D_{\mathrm A}^{u,\beta}$, so the norm-relative condition implies the directional one. Pointwise greedy-control stability gives
$\norm{\Acal(u)-\Tcal(u)}_{H^u}\le\varepsilon_{\mathrm A}(u)$; integrating yields \cref{eq:occupancy-map-distance-bound}. Minkowski's inequality gives
$R_{\mathrm{HJB}}^{u,\beta}\le D_{\mathrm A}^{u,\beta}+M_{\mathrm A}^{u,\beta}$ and
$D_{\mathrm A}^{u,\beta}\le R_{\mathrm{HJB}}^{u,\beta}+M_{\mathrm A}^{u,\beta}$, which imply \cref{eq:occupancy-residual-equivalence} and then \cref{eq:occupancy-relative-markov-improvement}.
\end{proof}

\begin{remark}[Directional versus norm-relative conditions]
\label{rem:relative-adjoint-scope}
The signed directional quantity in \cref{eq:occupancy-directional-defect} is the error component that enters the value increment. The norm-relative condition \cref{eq:occupancy-relative-population-score-condition} is stronger: besides directional improvement, it controls the distance between $\Acal$ and $\Tcal$ and therefore rules out asymptotic adjoint stationarity without HJB stationarity. The occupation formulation is intrinsic rather than cosmetic. The fixed-latent first-adjoint defect aggregates future residuals, so a current point may have zero policy residual while retaining a nonzero defect generated later along the path; a general pointwise defect-to-current-residual ratio is therefore not expected. The regular-tube envelope identity motivates residual-scaled decay but does not furnish an iterate-uniform global constant, especially across switching strata. Damping changes the admissible directional threshold $c_\beta$ but does not reduce the pre-damping population consistency error.
\end{remark}
\begin{remark}[Gain selected from the a posteriori lower bound]
\label{rem:certificate-gain}
At a fixed state and for a represented candidate held fixed as $\beta$ varies, set
\begin{equation}
  D_{\mathrm{step}}=\norm{d^u}_{H^u},
  \qquad E_{\mathrm{ham}}=\epsham(u),
  \qquad E_{\mathrm{rec}}=\epsrec(u).
\end{equation}
For $D_{\mathrm{step}}>0$, the sharp lower bound in \cref{eq:approx-score-gain-sharp} is the concave quadratic
\begin{equation}
  \mathcal C_{\mathrm{cert}}(\beta)
  =\beta\left[\frac{2-\beta}{2}D_{\mathrm{step}}^2-E_{\mathrm{ham}}D_{\mathrm{step}}-E_{\mathrm{rec}}\right],
\end{equation}
whose maximizer over $[0,1]$ is
\begin{equation}
  \beta_{\mathrm{cert}}^\star
  =\Pi_{[0,1]}\left(1-\frac{E_{\mathrm{ham}}}{D_{\mathrm{step}}}-\frac{E_{\mathrm{rec}}}{D_{\mathrm{step}}^2}\right).
  \label{eq:certificate-gain}
\end{equation}
Here $\Pi_{[0,1]}(x):=\min\{1,\max\{0,x\}\}$ is scalar clipping. Thus the exact local model $(E_{\mathrm{ham}}=E_{\mathrm{rec}}=0)$ favors the full step, whereas Hamiltonian approximation and local optimization error can make an interior gain optimal for the certificate. This is a pointwise conditional statement: the globally best gain may differ because the state law, field estimator, and error terms can themselves depend on $\beta$. In particular, the formula explains the possibility of an interior deterministic line-profile optimum but does not by itself predict the numerical value $\beta=1/2$.
\end{remark}

\section{Constrained dynamic portfolio choice}
\label{sec:portfolio}

\subsection{Adjoint and HJB portfolio Hamiltonians}
\label{sec:portfolio-decoder}

Let $Z_t\in\R^{d_z}$ be observable investment-opportunity factors and $X_t$ wealth. A generic multifactor portfolio model is
\begin{align}
  \dd Z_t&=K_Z(\bar z-Z_t)\dd t+\Lambda\dd B_t^Z,\\
  \frac{\dd X_t}{X_t}
  &=\bigl(r_f+u_t^\top\mu(Z_t)\bigr)\dd t
  +u_t^\top\Sigma_R\dd B_t^R,
  \label{eq:multifactor-portfolio}
\end{align}
where $K_Z$ is the factor mean-reversion matrix, $\bar z$ is the long-run factor mean, $\Lambda$ is the factor-volatility matrix, $r_f$ is the risk-free rate, $\mu(z)$ is the vector of excess returns, $\Sigma_R$ is the return-volatility loading, and $u_t\in\R^{d_u}$ is the risky-weight vector. Write
\begin{equation}
  \Sigma:=\Sigma_R\Sigma_R^\top\succ0
  \label{eq:return-covariance}
\end{equation}
for the instantaneous return covariance. If $\rho_{RZ}$ is the instantaneous Brownian cross-correlation matrix, so that
$\dd\langle B^R,B^Z\rangle_t=\rho_{RZ}\dd t$, define
\begin{equation}
  C:=\Sigma_R\rho_{RZ}\Lambda^\top\in\R^{d_u\times d_z}
  \label{eq:return-factor-covariance}
\end{equation}
for the return--factor cross-covariance.

Continuous-time portfolio choice under convex, no-short-sale, borrowing, and consumption constraints has classical duality and HJB foundations \citep{CvitanicKaratzas1992,Zariphopoulou1994,XuShreve1992b}. Simulation-based dynamic portfolio methods provide a complementary route for predictable returns, learning, and constraints \citep{BrandtEtAl2005}. The present focus is the continuous-time adjoint field needed for a pointwise constrained policy update and its self-consistent re-evaluation under the deployed feedback.

The generic adjoint-based policy update follows directly from the generalized Hamiltonian. Partition the second adjoint according to $S=(X,Z)$ and let $P_{ZX}^u$ be its factor--wealth block. Up to terms independent of a candidate risky weight $a$, the risky-weight part of the generalized Hamiltonian is
\begin{equation}
\begin{split}
  \Hgen_{\mathrm{port}}^u(t,x,z,a)
  ={}&x\Bigl[
  \lambda_X^u\mu(z)
  +\Sigma_R(\zeta_X^u)^\top
  +C P_{ZX}^u
  \Bigr]^\top a\\
  &+\frac12x^2P_{XX}^u\,a^\top\Sigma a.
  \label{eq:portfolio-adjoint-score}
\end{split}
\end{equation}
Thus, when $P_{XX}^u<0$, the risky-weight block is a strictly concave quadratic program conditional on the wealth-relevant adjoint tuple
\begin{equation}
  \vartheta_{\mathrm{port}}^u
  =(\lambda_X^u,\zeta_X^u,P_{XX}^u,P_{ZX}^u).
  \label{eq:portfolio-adjoint-tuple}
\end{equation}
No value-Hessian identification is required to form this QP. When $\lambda_X^u>0$, division by the positive factor $x\lambda_X^u$ gives a normalized adjoint Hamiltonian that can be compared directly with the HJB action Hamiltonian below. At the current optimal action, Equation~\eqref{eq:smooth-adjoint-identification} makes the shifted and second-adjoint terms reproduce the HJB action gradient, without requiring $P^\star=\D^2V$.

The exact policy-iteration theorem below uses the fixed-policy HJB action Hamiltonian and its greedy policy-improvement operator. For CRRA utility,
\begin{equation}
  V^u(t,x,z)=\frac{x^{1-\gamma}}{1-\gamma}F^u(t,z),
  \qquad \gamma>1,
  \label{eq:portfolio-crra-factorization}
\end{equation}
and the two ratios entering policy improvement are
\begin{equation}
  \frac{xV_{xx}^u}{V_x^u}=-\gamma,
  \qquad
  \frac{V_{xz}^u}{V_x^u}=\D_z\log F^u.
  \label{eq:portfolio-crra-ratios}
\end{equation}
Hence the wealth curvature is known analytically, while the unknown policy-dependent object is the normalized factor-adjoint field
\begin{equation}
  R^u(t,z):=\D_z\log F^u(t,z)\in\R^{d_z}.
  \label{eq:normalized-adjoint-field}
\end{equation}
After removing the positive multiplier $w^u=xV_x^u$, the normalized HJB action Hamiltonian is
\begin{equation}
  q^u(t,z,a)
  =c_0^u(t,z)+a^\top m^u(t,z)-\frac\gamma2a^\top\Sigma a,
  \qquad
  m^u(t,z)=\mu(z)+C R^u(t,z),
  \qquad H^u=\gamma\Sigma.
  \label{eq:portfolio-score}
\end{equation}
For a positive-definite matrix $M$, let
\begin{equation}
  \Pi_{\mathcal C}^{M}(x)
  :=\argmin_{a\in\mathcal C}\frac12\norm{a-x}_{M}^2
  \label{eq:metric-projection}
\end{equation}
be the $M$-metric projection onto a closed convex set $\mathcal C$. The exact HJB policy-improvement operator is
\begin{equation}
  \Tcal(u)(t,z)
  =\Pi_{\U(t,z)}^{\Sigma}
  \left(
  \frac1\gamma\Sigma^{-1}\bigl[\mu(z)+C R^u(t,z)\bigr]
  \right).
  \label{eq:portfolio-recovery}
\end{equation}

The generic adjoint Hamiltonian in \cref{eq:portfolio-adjoint-score} is useful for the broader controlled-diffusion interpretation. For the affine-return CRRA iteration, write $\mu(z)=\mu_c+B_\mu z$, with intercept $\mu_c$ and loading matrix $B_\mu$, and define the population operator from the reduced fixed-policy equation. Let $\chi=1-\gamma$ and, for a fixed feedback $u$, write
\begin{equation}
  \widetilde b^u(t,z)
  :=K_Z(\bar z-z)+\chi C^\top u(t,z),\qquad
  c^u(t,z):=\chi\left(r_f+u^\top\mu(z)-\frac\gamma2u^\top\Sigma u\right).
  \label{eq:reduced-drift-potential}
\end{equation}
Here $c^u$ is the Feynman--Kac potential, distinct from the action-independent Hamiltonian term $c_0^u$. Let $Q_u$ denote the factor law with drift $\widetilde b^u$ and diffusion $\Lambda$. Then
\[
  F^u(t,z)=\E_{t,z}^{Q_u}\exp\left\{\int_t^T c^u(r,Z_r)\dd r\right\}.
\]
Fixed-latent differentiation holds the deployed action $u(t,Z_t)$ fixed when differentiating this representation. In the affine model the corresponding factor tangent is deterministic,
$J_{\mathrm{FL},r}^u=e^{-K_Z(r-t)}$, and
\begin{equation}
  G_{\mathrm{OL}}^u(t,z)
  :=\E_{t,z}^{Q_u}\left[
  e^{\int_t^T c^u(q,Z_q)\dd q}
  \int_t^T (J_{\mathrm{FL},r}^u)^\top
  \partial_z c^u(r,Z_r;u\ \mathrm{fixed})\dd r
  \right],
  \qquad
  R_{\mathrm{OL}}^u:=\frac{G_{\mathrm{OL}}^u}{F^u}.
  \label{eq:crra-reduced-ol-target}
\end{equation}
In this reduced formulation, $G_{\mathrm{OL}}^u\in\R^{d_z}$ is the population reverse-mode derivative of the reduced continuation payoff with all future deployed action values detached. Since
$\partial_z c^u(r,z;u\ \mathrm{fixed})=\chi B_\mu^\top u(r,z)$ and the diffusion is state independent, \cref{eq:crra-reduced-ol-target} is a complete population definition of the field estimated by reduced OL-BPTT.

With action-independent term $c_{\mathrm A}^u$, the population normalized adjoint Hamiltonian and policy-improvement operator are
\begin{align}
  q_{\mathrm A}^u(t,z,a)
  &=c_{\mathrm A}^u(t,z)+a^\top\bigl[\mu(z)+C R_{\mathrm{OL}}^u(t,z)\bigr]
    -\frac\gamma2a^\top\Sigma a,
  \label{eq:portfolio-olbptt-score}\\
  \Acal(u)(t,z)
  &=\Pi_{\U(t,z)}^{\Sigma}
  \left(
  \frac1\gamma\Sigma^{-1}\bigl[\mu(z)+C R_{\mathrm{OL}}^u(t,z)\bigr]
  \right).
  \label{eq:portfolio-olbptt-recovery}
\end{align}
Hence the reduced Hamiltonian-gradient discrepancy is exactly
\begin{equation}
  e_{\mathrm{cons}}^u(t,z,a)
  =C\bigl(R_{\mathrm{OL}}^u(t,z)-R^u(t,z)\bigr),
  \label{eq:portfolio-reduced-compatibility}
\end{equation}
independently of the trial action. The curvature $\gamma\Sigma$ and the feasible set are common to $\Acal$ and $\Tcal$. Thus \cref{eq:portfolio-reduced-compatibility} is the entire population adjoint--HJB discrepancy in this reduced CRRA formulation.

For completeness, the multiplicative-potential analogue of the fixed-latent envelope identity can be written directly. Let $p_F^u:=\D_zF^u$ and define the scalar factor generator
\begin{equation}
  \Lcal_Z^u\phi
  :=\widetilde b^u\!\cdot\D_z\phi
  +\frac12\tr(\Lambda\Lambda^\top\D_{zz}^2\phi).
  \label{eq:crra-reduced-factor-generator}
\end{equation}
Let $\Lcal_Z^u$ act componentwise on vector fields and set
\begin{equation}
  \mathscr M_{\mathrm{FL}}^u h
  :=\partial_t h+\Lcal_Z^u h-K_Z^\top h+c^u h.
  \label{eq:crra-reduced-vector-operator}
\end{equation}
This is the vector backward operator obtained from \cref{eq:crra-global-fixed-policy} by differentiating $\widetilde b^u$ and $c^u$ while holding the deployed action fixed. Then
\begin{equation}
  \mathscr M_{\mathrm{FL}}^u\bigl(p_F^u-G_{\mathrm{OL}}^u\bigr)
  +r_{\mathrm{env,red}}^u=0,
  \qquad
  \bigl(p_F^u-G_{\mathrm{OL}}^u\bigr)(T,\cdot)=0,
  \label{eq:crra-reduced-defect-pde}
\end{equation}
with source
\begin{equation}
  r_{\mathrm{env,red}}^u
  =\chi F^u(\D_zu)^\top
  \bigl[\mu+C R^u-\gamma\Sigma u\bigr].
  \label{eq:crra-reduced-envelope-source}
\end{equation}
On a regular interior or fixed-face cell, KKT normal--tangent cancellation gives the contracted identity
\begin{equation}
  (\D_zu)^\top\bigl[\mu+C R^u-\gamma\Sigma u\bigr]
  =(\D_zu)^\top\gamma\Sigma\bigl[\Tcal(u)-u\bigr].
  \label{eq:crra-reduced-contracted-residual}
\end{equation}
Consequently,
\begin{equation}
  \norm{r_{\mathrm{env,red}}^u}
  \le |\chi|F^u
  \norm{(\D_zu)^\top(\gamma\Sigma)^{1/2}}_{\mathrm{op}}
  \norm{\Tcal(u)-u}_{\gamma\Sigma}.
  \label{eq:crra-reduced-envelope-bound}
\end{equation}
The vector Feynman--Kac representation for $\mathscr M_{\mathrm{FL}}^u$ therefore yields the reduced counterpart of \cref{eq:olbptt-residual-scaled-bound}. In particular, on a regular cell,
$\Tcal(u)=u$ implies $G_{\mathrm{OL}}^u=\D_zF^u$, hence $R_{\mathrm{OL}}^u=R^u$ and $\Acal(u)=u$. This fixed-stratum conclusion does not extend across switching surfaces; the directional occupation and asymptotic stationarity-compatibility conditions used below remain separate global hypotheses, with the norm-relative condition serving as a stronger sufficient route.

This reduction changes the represented statistic, not the differentiation convention. It does not identify the PMP second adjoint with the value Hessian or impose a zero shifted martingale input. Rather, homotheticity makes the structural curvature known and permits a problem with fifty assets and three factors to be represented by a three-dimensional state-dependent field followed by a fifty-dimensional QP.

\subsection{Global convergence for constrained CRRA portfolios}
\label{sec:crra-global-convergence}

Let $\chi:=1-\gamma<0$, suppose $\mu(z)=\mu_c+B_\mu z$, and let the admissible portfolio set be a nonempty compact convex set $\PortSet\subset\R^{d_u}$ independent of $(t,z)$. For a fixed $\PortSet$-valued Markov policy $u$, the reduced factor solves
\begin{equation}
\begin{split}
  0={}&\partial_tF^u+
  \bigl[K_Z(\bar z-z)+\chi C^\top u\bigr]^\top\D_zF^u
  +\frac12\tr\!\left(\Lambda\Lambda^\top\D^2_{zz}F^u\right)\\
  &+\chi\left(r_f+u^\top\mu(z)-\frac\gamma2u^\top\Sigma u\right)F^u,
  \qquad F^u(T,z)=1.
  \label{eq:crra-global-fixed-policy}
\end{split}
\end{equation}
For a positive factor $F$, define
\begin{equation}
  \mathcal G[F](t,z)
  :=\Pi_{\PortSet}^\Sigma\!\left(
  \frac1\gamma\Sigma^{-1}\bigl[\mu(z)+C\D_z\log F(t,z)\bigr]
  \right),
  \label{eq:crra-global-recovery}
\end{equation}
so that $\Tcal(u)=\mathcal G[F^u]$. The reduced HJB equation is
\begin{equation}
\begin{split}
  0={}&\partial_tF+K_Z(\bar z-z)^\top\D_zF
  +\frac12\tr\!\left(\Lambda\Lambda^\top\D^2_{zz}F\right)\\
  &+\chi\sup_{a\in\PortSet}\left\{
  \left(r_f+a^\top\mu(z)-\frac\gamma2a^\top\Sigma a\right)F
  +a^\top C\D_zF
  \right\},
  \qquad F(T,z)=1.
  \label{eq:crra-global-hjb}
\end{split}
\end{equation}
The positive multiplier in the performance-difference identity is
\begin{equation}
  w^u(t,x,z)=x^\chi F^u(t,z).
  \label{eq:crra-positive-multiplier}
\end{equation}

For $\eta>0$, write
\begin{equation}
  \mathcal C_\eta:=\left\{f:\ \sup_{(t,z)}e^{-\eta\langle z\rangle}|f(t,z)|<\infty\right\},
  \qquad \langle z\rangle:=(1+|z|^2)^{1/2}.
  \label{eq:crra-weighted-class}
\end{equation}
Set
\begin{equation}
  g(z,a):=r_f+a^\top\mu(z)-\frac\gamma2a^\top\Sigma a,
  \qquad
  \Hgen_{\mathrm B}(z,F,q):=\chi\sup_{a\in\PortSet}\{g(z,a)F+a^\top Cq\}.
  \label{eq:bellman-hamiltonian}
\end{equation}
Compactness of $\PortSet$ and affinity of $\mu$ give constants $C_0,C_1,L_q<\infty$ such that
\begin{equation}
  \sup_{a\in\PortSet}|g(z,a)|\le C_0+C_1|z|,
  \qquad
  |\Hgen_{\mathrm B}(z,F,q_1)-\Hgen_{\mathrm B}(z,F,q_2)|\le L_q|q_1-q_2|.
  \label{eq:hamiltonian-growth}
\end{equation}
For $p>d_z+2$ and $\eta_0>0$, let $\mathcal S_{\eta_0}$ be the positive functions in
$C([0,T]\times\R^{d_z})\cap W_{p,\mathrm{loc}}^{2,1}([0,T)\times\R^{d_z})\cap\mathcal C_{\eta_0}$.

\begin{assumption}[CRRA coefficients and warm start]
\label{ass:crra-global}
The portfolio set $\PortSet$ is compact and convex, $K_Z$ is positive stable, $\Lambda\Lambda^\top\succ0$, and $\mu$ is affine. The warm start $u^0$ is a $\PortSet$-valued Borel Markov feedback and, for some $\alpha_0\in(0,1)$, has a local parabolic $C^{\alpha_0/2,\alpha_0}$ modulus on every compact cylinder.
\end{assumption}

For $0<\tau<T$ and $R>0$, set $K_{\tau,R}:=[0,T-\tau]\times\{z\in\R^{d_z}:|z|\le R\}$.

\begin{proposition}[Policy-uniform regularity, invariant iteration class, and short-time occupation]
\label{prop:crra-fixed-policy-regularity}
Under \cref{ass:crra-global}, let $u:[0,T]\times\R^{d_z}\to\PortSet$ be any bounded Borel Markov feedback. The Feynman--Kac functional associated with \cref{eq:crra-global-fixed-policy} is the unique positive $W_{p,\mathrm{loc}}^{2,1}$ strong solution in an exponential weighted class. For every compact cylinder $K_{\tau,R}$ there are policy-independent constants such that, on the buffered cylinder $K_{\tau/2,R+1}$,
\begin{equation}
  0<m_{\tau,R}\le F^u\le M_{\tau,R}<\infty,
  \qquad
  \norm{F^u}_{W_p^{2,1}(K_{\tau/2,R+1})}
  +\norm{G_{\mathrm{OL}}^u}_{W_p^{2,1}(K_{\tau/2,R+1})}
  \le C_{\tau,R},
  \label{eq:crra-policy-uniform-Wp}
\end{equation}
for the fixed $p>d_z+2$ above. Hence, for every
$0<\bar\alpha<1-(d_z+2)/p$, the maps $\mathcal G[F^u]$ and $\Acal(u)$ have a common policy-uniform local
$C^{\bar\alpha/2,\bar\alpha}$ modulus on compact cylinders. Fix one such $\bar\alpha$ and set $\alpha_\star:=\min\{\alpha_0,\bar\alpha\}$. The exact and population damped updates generated from $u^0$ therefore remain in an invariant class $\mathfrak U$ with a common local $C^{\alpha_\star/2,\alpha_\star}$ modulus. For every successive iterate pair $(u,v)$ produced by either recursion, the performance-difference identity \cref{eq:performance-difference} is valid.

Moreover, for every spatial radius $\varrho\in(0,1]$ there is
$h_0=h_{\tau,R}(\varrho)\in(0,\tau/2]$ such that, for every $h\in(0,h_0]$, every $(t,z)\in K_{\tau,R}$, every bounded Borel Markov evaluation feedback $u$ taking values in $\PortSet$, and every progressively measurable $\PortSet$-valued deployment control $v$,
\begin{equation}
  \E_{t,1,z}^{v}\left[
  \int_t^{t+h}w^u(s,S_s^v)
  \mathbf 1_{\{|Z_s^v-z|\le\varrho\}}\dd s
  \right]
  \ge c_{\tau,R}(h):=\frac12h\,m_{\tau,R}\,2^\chi>0.
  \label{eq:crra-scale-coverage}
\end{equation}
\end{proposition}

\begin{proof}
For a bounded Borel Markov feedback $u$, the $Q_u$-factor dynamics are an Ornstein--Uhlenbeck (OU) diffusion with a bounded drift perturbation. Since $\Lambda\Lambda^\top\succ0$, the kernel
$\theta_u(t,z):=\Lambda^\top(\Lambda\Lambda^\top)^{-1}\chi C^\top u(t,z)$ is bounded. Girsanov transformation gives a unique weak law $Q_u$ from the Gaussian OU reference law \citep{KaratzasShreve1991}, and variation of constants gives policy-uniform exponential moments. With
$A_u:=\int_t^Tc^u(r,Z_r)\dd r$, the growth bound \cref{eq:hamiltonian-growth}, Jensen's inequality, and the exponential moments yield common constants $C,\eta_0>0$ such that
\begin{equation}
  C^{-1}e^{-\eta_0\langle z\rangle}
  \le F^u(t,z)\le Ce^{\eta_0\langle z\rangle}
  \qquad\text{for all bounded Borel Markov }u.
  \label{eq:crra-two-sided-exponential-envelope}
\end{equation}
Thus $F^u$ is bounded above and away from zero on compact cylinders, uniformly in $u$. Local $W_p^{2,1}$ solvability with a constant uniformly elliptic principal part and bounded measurable lower-order coefficients applies after truncation and localization \citep{KimKrylov2007}. It\^o--Krylov identifies the stopped strong solutions with their Feynman--Kac expectations; the exponential moments remove the stopping, and the local estimates give the first bound in \cref{eq:crra-policy-uniform-Wp}. The linear Lyapunov-barrier comparison in Appendix~\ref{app:bellman-identification} gives weighted uniqueness \citep{Krylov1987,Friedman1964,FlemingSoner2006,Lieberman1996}.

The reduced fixed-latent field solves the linear vector equation
\begin{equation}
\begin{split}
  0={}&\partial_tG_{\mathrm{OL}}^u+\Lcal_Z^uG_{\mathrm{OL}}^u
  -K_Z^\top G_{\mathrm{OL}}^u+c^uG_{\mathrm{OL}}^u
  +\chi F^uB_\mu^\top u,\\
  G_{\mathrm{OL}}^u(T,\cdot)={}&0.
\end{split}
\label{eq:crra-reduced-ol-pde}
\end{equation}
Its coefficients and source are uniformly bounded on buffered compact cylinders, so the same interior estimate gives the second bound in \cref{eq:crra-policy-uniform-Wp}. The compact parabolic embedding
\begin{equation}
  W_p^{2,1}(K_{\tau/2,R+1})
  \Subset C^{(1+\bar\alpha)/2,1+\bar\alpha}(K_{\tau,R}),
  \qquad 0<\bar\alpha<1-\frac{d_z+2}{p},
  \label{eq:crra-compact-embedding}
\end{equation}
controls $\D_zF^u$ and $G_{\mathrm{OL}}^u$ in a common local H\"older class. Since $F^u\ge m_{\tau,R}$ and metric projection onto $\PortSet$ is Lipschitz, both $\mathcal G[F^u]$ and $\Acal(u)$ have policy-independent compactwise moduli. For each $K_{\tau,R}$, let $L_{\tau,R}^\star$ be the maximum of the resulting $C^{\alpha_\star/2,\alpha_\star}$ bound and the corresponding seminorm of $u^0$, and let $\mathfrak U$ be the class of $\PortSet$-valued feedbacks whose seminorm on every $K_{\tau,R}$ is at most $L_{\tau,R}^\star$. Convex damping preserves these bounds, proving invariance of $\mathfrak U$ for both recursions.

For $u\in\mathfrak U$, the lower-order coefficients are locally H\"older. Local Schauder regularity upgrades $F^u$ to $C^{1+\alpha_\star/2,2+\alpha_\star}$ away from $T$. Applying It\^o's formula after localization, and using bounded controls together with the exponential factor moments and negative wealth moments, proves \cref{eq:performance-difference} for successive iterate pairs.

For the occupation estimate, fix $(t,z)\in K_{\tau,R}$, start from $X_t^v=1$, set $L_s^X:=\log X_s^v$, and let
\[
  \tau_{\varrho}:=\inf\{s\ge t:\ |Z_s^v-z|>\varrho\}.
\]
On $[t,\tau_{\varrho}]$, the factor remains in $\{|y|\le R+1\}$. Compactness of $\PortSet$ and affinity of $\mu$ make the stopped drift and diffusion coefficients of both $Z^v-z$ and $L^X$ uniformly bounded. The Burkholder--Davis--Gundy inequality therefore gives constants $C_Z,C_X<\infty$, independent of $u,v,t,z$, such that for $h\le1$,
\begin{equation}
\begin{split}
  \E_{t,1,z}^v\!\left[\sup_{t\le s\le t+h}
  |Z_{s\wedge\tau_{\varrho}}^v-z|^2\right]&\le C_Zh,\\
  \E_{t,1,z}^v\!\left[\sup_{t\le s\le t+h}
  |L^X_{s\wedge\tau_{\varrho}}|^2\right]&\le C_Xh.
\end{split}
\label{eq:crra-coverage-stopped-moments}
\end{equation}
Hence
\[
  \Pp_{t,1,z}^v(\tau_{\varrho}\le t+h)
  \le \frac{C_Zh}{\varrho^2},
  \qquad
  \Pp_{t,1,z}^v\!\left(
    \sup_{t\le s\le t+h}X_s^v>2,\ \tau_{\varrho}>t+h
  \right)
  \le \frac{C_Xh}{(\log2)^2}.
\]
Choose $h_0\le\min\{\tau/2,1\}$ so that the sum of these bounds is at most $1/2$. On the complementary event, $|Z_s^v-z|\le\varrho$ and $X_s^v\le2$ on $[t,t+h]$. Since $\chi<0$ and $F^u\ge m_{\tau,R}$ there,
\[
  w^u(s,S_s^v)=(X_s^v)^\chi F^u(s,Z_s^v)
  \ge2^\chi m_{\tau,R}.
\]
The good event has probability at least $1/2$, and integration gives \cref{eq:crra-scale-coverage}.
\end{proof}

\begin{proposition}[Weighted comparison and terminal attachment]
\label{prop:crra-analytic-closure}
Under \cref{ass:crra-global}, two strong solutions in $\mathcal S_{\eta_0}$ of \cref{eq:crra-global-hjb} with the same terminal data coincide. Moreover, for every $R<\infty$ there is $C_R<\infty$ such that
\begin{equation}
  \sup_{u:\,[0,T]\times\R^{d_z}\to\PortSet\ \mathrm{Borel}}
  \sup_{|z|\le R}|F^u(t,z)-1|\le C_R(T-t),
  \qquad t\in[\max\{0,T-1\},T].
  \label{eq:terminal-attachment}
\end{equation}
Thus every locally uniform limit of fixed-policy factors has terminal value one. The weighted comparison uses an exponential Lyapunov barrier to absorb the linearly growing zero-order coefficient and the parabolic ABP--Krylov maximum principle for $W_p^{2,1}$ strong solutions.
\end{proposition}

\begin{proof}
See Appendix~\ref{app:crra-global-proof}.
\end{proof}

\begin{theorem}[Global convergence of exact HJB policy iteration]
\label{thm:crra-global-convergence}
Under \cref{ass:crra-global}, the reduced HJB equation \eqref{eq:crra-global-hjb} has a unique positive strong solution $F^\star$ in the exponential weighted class generated by the iteration. The feedback $u^\star=\mathcal G[F^\star]$ is optimal for the constrained portfolio problem and
\[
  V^\star(t,x,z)=\frac{x^\chi}{\chi}F^\star(t,z).
\]
For the warm start $u^0$ in \cref{ass:crra-global} and gains $0<\underline\beta\le\beta_k\le\overline\beta\le1$, the exact iteration
\begin{equation}
  u^{k+1}=(1-\beta_k)u^k+\beta_k\mathcal G[F^{u^k}]
  \label{eq:crra-exact-damped-iteration}
\end{equation}
satisfies
\begin{equation}
  V^{u^k}(t,x,z)\uparrow V^\star(t,x,z),
  \label{eq:crra-value-global-convergence}
\end{equation}
$F^{u^k}\to F^\star$ locally in $C^{(1+\bar\alpha)/2,1+\bar\alpha}$ for every $0<\bar\alpha<1-(d_z+2)/p$, and
\begin{equation}
  u^k\to u^\star,\qquad
  \norm{\Tcal(u^k)-u^k}_{L^\infty(K_{\tau,R})}\to0
  \label{eq:crra-policy-global-convergence}
\end{equation}
locally uniformly on $[0,T)\times\R^{d_z}$. No contraction of $\Tcal$ is assumed.
\end{theorem}

\begin{proof}
Exact value improvement makes $F^{u^k}$ monotone, and Proposition~\ref{prop:crra-fixed-policy-regularity} supplies the invariant class, performance-difference regularity, and locally strong subsequential compactness. The short-time occupation bound upgrades integrated residual summability to local uniform residual decay, which identifies every limit as a strong solution of the reduced HJB equation. Proposition~\ref{prop:crra-analytic-closure} supplies the terminal datum and uniqueness. The It\^o--Krylov/Girsanov verification then identifies the common limit with the constrained financial value and its greedy feedback with an optimizer. Appendix~\ref{app:crra-global-proof} gives the complete comparison, limit, and verification argument.
\end{proof}

For gains $0<\underline\beta\le\beta_k\le\overline\beta\le1$, consider the population OL-BPTT recursion
\begin{equation}
  u^{k+1}=(1-\beta_k)u^k+\beta_k\Acal(u^k).
  \label{eq:crra-adjoint-damped-iteration}
\end{equation}

\begin{assumption}[Directional improvement and asymptotic stationarity compatibility]
\label{ass:crra-adjoint-compatible}
Set
\begin{equation}
  c_{\overline\beta}:=\frac{2-\overline\beta}{2},
  \qquad \kappa_{\mathrm{dir}}\in\R,\quad \kappa_{\mathrm{dir}}<c_{\overline\beta}.
  \label{eq:crra-directional-kappa}
\end{equation}
For the actual sequence in \cref{eq:crra-adjoint-damped-iteration}, the following two conditions hold.

\emph{(A1) Directional improvement.} For every $k$ and starting state,
\begin{equation}
  \Xi_{\mathrm A}^{u^k,\beta_k}
  \le \kappa_{\mathrm{dir}}
  \bigl(D_{\mathrm A}^{u^k,\beta_k}\bigr)^2.
  \label{eq:crra-directional-alignment}
\end{equation}

\emph{(A2) Asymptotic stationarity compatibility.} For every compact cylinder $K_{\tau,R}$, every sequence $k_j\to\infty$, and every moving sequence $(t_j,z_j)\in K_{\tau,R}$,
\begin{equation}
  D_{\mathrm A}^{u^{k_j},\beta_{k_j}}(t_j,1,z_j)\longrightarrow0
  \quad\Longrightarrow\quad
  M_{\mathrm A}^{u^{k_j},\beta_{k_j}}(t_j,1,z_j)\longrightarrow0.
  \label{eq:crra-stationarity-compatibility}
\end{equation}
The regularity and invariant-class conclusions needed by the iteration follow from Proposition~\ref{prop:crra-fixed-policy-regularity}.
\end{assumption}

\begin{remark}[Strong norm-relative sufficient condition and auditability]
\label{rem:crra-strong-relative-route}
A common constant $0\le\kappa<c_{\overline\beta}$ satisfying, for the actual iterates and starts,
\begin{equation}
  E_{\mathrm A}^{u^k,\beta_k}
  \le\kappa D_{\mathrm A}^{u^k,\beta_k}
  \label{eq:crra-relative-adjoint-condition}
\end{equation}
implies both parts of \cref{ass:crra-adjoint-compatible}: Cauchy--Schwarz gives \cref{eq:crra-directional-alignment} with $\kappa_{\mathrm{dir}}=\kappa$, while $M_{\mathrm A}\le E_{\mathrm A}$ gives \cref{eq:crra-stationarity-compatibility}. Thus the norm-relative condition is a convenient strong sufficient condition, not the primitive convergence hypothesis. Analytic verification propositions naturally control $E_{\mathrm A}$ through the defect equation, whereas the theorem itself only requires the map-distance version in \cref{eq:crra-stationarity-compatibility}. For a full step the sufficient norm-relative threshold is $1/2$; for the fixed half-step protocol it is $3/4$. The direct ratio $\widehat\kappa_{\mathrm{dir}}=\widehat\Xi_{\mathrm A}/\widehat D_{\mathrm A}^{\,2}$ and the stronger norm ratio $\widehat\kappa_{\mathrm{occ}}=\widehat E_{\mathrm A}/\widehat D_{\mathrm A}$ are estimated under the updated-policy law in \cref{tab:kappa-occupation-audit}. The directional ratio audits (A1); a uniform norm-relative bound implies both hypotheses. Both error ratios require an exact or independently estimated HJB reference, whereas the update field $\Acal(u)-u$ is available from the algorithm.
\end{remark}

\begin{theorem}[Global convergence of population OL-BPTT policy iteration]
\label{thm:crra-adjoint-global}
Suppose \cref{ass:crra-global,ass:crra-adjoint-compatible} hold for \cref{eq:crra-adjoint-damped-iteration}. Then, for every starting state,
\begin{equation}
\begin{split}
  V^{u^{k+1}}-V^{u^k}
  \ge{}&c_{\mathrm A}
  \E^{u^{k+1}}\left[
  \int w^{u^k}
  \norm{\Acal(u^k)-u^k}_{\gamma\Sigma}^2\dd r
  \right],\\
  c_{\mathrm A}:={}&
  \underline\beta(c_{\overline\beta}-\kappa_{\mathrm{dir}})>0.
  \label{eq:crra-adjoint-value-increment}
\end{split}
\end{equation}
In particular, the weighted population OL-BPTT movements are square summable for every fixed start. The value, factor, and policy conclusions of \cref{thm:crra-global-convergence} hold for this sequence, and on every compact cylinder
\begin{equation}
  \norm{\Tcal(u^k)-u^k}_\infty
  +\norm{\Acal(u^k)-u^k}_\infty
  +\norm{\Acal(u^k)-\Tcal(u^k)}_\infty
  \longrightarrow0.
  \label{eq:crra-adjoint-residual-convergence}
\end{equation}
No contraction of either operator is assumed.
\end{theorem}

\begin{proof}
The directional part of Corollary~\ref{cor:occupancy-relative-adjoint-step} gives \cref{eq:crra-adjoint-value-increment}. Monotonicity and the invariant-class bounds yield a locally uniform factor limit, so consecutive value increments vanish locally uniformly; the displayed bound then forces the occupation movement $D_{\mathrm A}^{u^k,\beta_k}$ to vanish uniformly over compact sets of starts. The moving-start compatibility condition \cref{eq:crra-stationarity-compatibility} transfers this decay to $M_{\mathrm A}^{u^k,\beta_k}$. The common local moduli and short-time occupation bound in Proposition~\ref{prop:crra-fixed-policy-regularity} upgrade these two occupation statements to local uniform decay of $\Acal(u^k)-u^k$ and $\Acal(u^k)-\Tcal(u^k)$; their triangle sum gives the HJB residual decay. Proposition~\ref{prop:crra-analytic-closure} attaches the terminal condition, and the exact HJB identification argument yields the stated value, factor, and policy limits. Appendix~\ref{app:crra-adjoint-global-proof} gives the complete moving-restart and limit argument.
\end{proof}
\subsection{Portfolio benchmark hierarchy}

\paragraph{Merton and one-factor predictable returns.}

With constant excess-return vector $\mu_0$, the unconstrained Merton optimizer is policy independent \citep{Merton1969}:
\begin{equation}
  u^\star=\gamma^{-1}\Sigma^{-1}\mu_0.
  \label{eq:merton-solution}
\end{equation}
Under convex trading bounds, its metric projection is still policy independent, so the model is a negative control for iteration.

For a scalar OU predictor,
\begin{align}
  \dd Y_t&=\kappa_Y(\bar y-Y_t)\dd t+\nu\dd B_t^Y,\\
  \frac{\dd X_t}{X_t}
  &=\{r_f+u_t(\mu_0+\mu_1Y_t)\}\dd t+u_t\sigma\dd B_t^R,
  \label{eq:predictable-dynamics}
\end{align}
where $\kappa_Y>0$ is the predictor mean-reversion rate, $\bar y$ its long-run mean, $\nu$ and $\sigma$ are factor and return volatilities, $\mu(y):=\mu_0+\mu_1y$, and $\dd\langle B^R,B^Y\rangle_t=\rho\dd t$. For a fixed feedback $u=u(t,y)$, homotheticity gives
\begin{equation}
  V^u(t,x,y)=\frac{x^{1-\gamma}}{1-\gamma}F^u(t,y),
\end{equation}
and the fixed-policy PDE
\begin{equation}
\begin{split}
  0={}&F_t^u+\bigl[\kappa_Y(\bar y-y)+(1-\gamma)u\sigma\nu\rho\bigr]F_y^u
  +\frac{\nu^2}{2}F_{yy}^u\\
  &+(1-\gamma)\left[r_f+u\mu(y)-\frac\gamma2u^2\sigma^2\right]F^u,
  \qquad F^u(T,y)=1.
  \label{eq:fixed-policy-pde}
\end{split}
\end{equation}
The exact HJB policy-improvement operator, which identifies $\Tcal$ for this model, is
\begin{equation}
  \Tcal_{\mathrm{PDE}}(u)(t,y)
  =\frac{\mu(y)+\rho\sigma\nu\,\partial_y\log F^u(t,y)}{\gamma\sigma^2}.
  \label{eq:pde-map}
\end{equation}

\paragraph{One-factor policy dependence.}
Fix a bounded policy class $\mathcal P$ and a norm $\norm{\cdot}_{\Hpol}$ for which pointwise box projection is nonexpansive. Suppose the fixed-policy PDE satisfies the gradient-sensitivity estimate
\begin{equation}
  \norm{\partial_y\log F^u-\partial_y\log F^v}_{\Hpol}
  \le C_{\log}(T;\mathcal P)\norm{u-v}_{\Hpol},
  \qquad u,v\in\mathcal P.
  \label{eq:logF-sensitivity}
\end{equation}
Subtracting \cref{eq:pde-map} for two policies and using nonexpansiveness of metric projection gives
\begin{equation}
  \norm{\Tcal(u)-\Tcal(v)}_{\Hpol}
  \le L_{\Tcal}(T;\mathcal P)\norm{u-v}_{\Hpol},
  \qquad
  L_{\Tcal}(T;\mathcal P)=\frac{|\rho|\nu}{\gamma\sigma}C_{\log}(T;\mathcal P).
  \label{eq:onefactor-map-constant}
\end{equation}
In particular, $\rho=0$ implies $L_{\Tcal}(T;\mathcal P)=0$: the operator is policy independent and the ideal undamped update reaches its recovered policy in one step. If a parabolic stability estimate gives $C_{\log}(T;\mathcal P)\le C_{\mathcal P}T$ on a short horizon, then $L_{\Tcal}(T;\mathcal P)=O(|\rho|\nu T)$.

The unconstrained analytical optimum is affine in $y$ \citep{KimOmberg1996,Liu2007}. A box constraint creates lower, interior, and upper action regions whose boundaries can be computed by a one-dimensional HJB reference.

\section{Portfolio experiments and adjoint--HJB consistency audits}
\label{sec:experiments}

Unless stated otherwise, results average three independent outer seeds; paired comparisons use common random numbers \citep{Glasserman2003}. Policy root-mean-square error (RMSE) uses the stated reference and evaluation law; on-policy evaluation uses the factor-occupancy law. Active-set error measures constraint-status mismatch; the KKT residual measures local first-order optimality violation. Nominal budgets count post-warm-start path-equivalent operator simulations, excluding the common actor warm start and deterministic regression/QP overhead. \emph{One-shot} evaluates under $u^0$ once; \emph{current-policy re-evaluation} evaluates under the updated rollout; \emph{pooled initial-policy refinement} combines estimates under $u^0$ using the stated budget and averaging rule.

\subsection{When should the policy-improvement operator be re-evaluated?}
\label{sec:when-iterate}

\paragraph{Merton negative control.}
In the one-hundred-asset no-short Merton benchmark the exact improvement operator is policy independent. The warm-start actor has RMSE $7.61\times10^{-3}$; one standard update lowers it to $2.53\times10^{-3}$. With six update-equivalent budgets, a single concentrated estimate reaches $6.00\times10^{-4}$, compared with $1.09\times10^{-3}$ for pooled initial-policy refinement and $1.62\times10^{-3}$ for current-policy iteration. The corresponding KKT residuals are $1.08\times10^{-3}$, $1.41\times10^{-3}$, and $1.94\times10^{-3}$, respectively. Iteration is therefore not a universal variance-reduction device.

\paragraph{Predictable-return mechanism.}
Starting from the myopic policy, RMSE falls from $5.99\times10^{-2}$ to $2.67\times10^{-3}$ after one update. Re-evaluating under that policy yields a second raw candidate with mean RMSE $6.58\times10^{-4}$ across the same three seeds (\cref{fig:reevaluation}). This smooth one-factor diagnostic uses common-random-number Monte Carlo log-certainty-equivalent (log-CE) estimates for Armijo backtracking before deployment: a central finite-difference slope, sufficient-increase coefficient $10^{-4}$, step halving down to $2^{-12}$, and a monotone-improvement fallback. This safeguard is separate from the positive-gain damped recursion analysed in Sections~\ref{sec:theory}--\ref{sec:portfolio}. The candidate is accepted in one seed; the other two have negative estimated directional slopes and retain the previous policy. The deployed-policy mean is $2.003\times10^{-3}$. Reaching the second candidate uses two standard operator passes, whereas concentrating eight passes' budget at the initial policy leaves RMSE near $2.68\times10^{-3}$. Re-evaluation thus adds policy-dependent information, although candidate accuracy and acceptance are distinct. A deterministic fixed-policy PDE profile favors a gain near one half (Appendix~\ref{app:numerical-details}).

\begin{figure}[!htbp]
\centering
\includegraphics[width=0.72\textwidth]{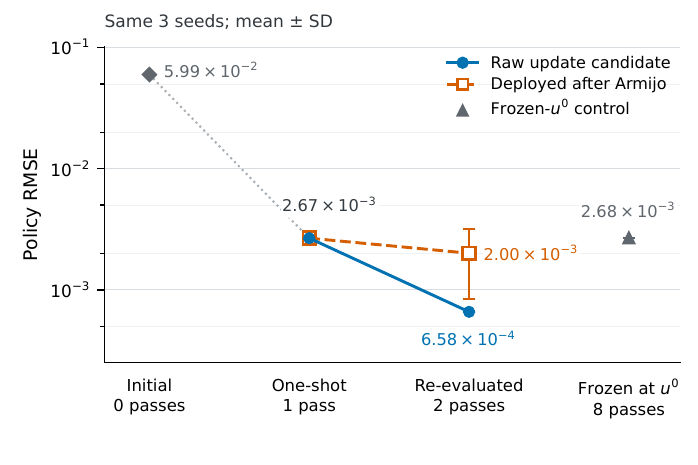}
\caption{Policy-dependent re-evaluation against the analytical reference. Means and one-standard-deviation bars use seeds $260730$, $260731$, and $260732$. Raw candidates and policies after Armijo acceptance are distinguished. The frozen-$u^0$ control uses eight nominal operator passes, versus two for re-evaluation.}
\label{fig:reevaluation}
\end{figure}

\subsection{Population consistency and occupation-measure audits}
\label{sec:population-audits}

The constrained experiments use the CRRA-reduced fixed-latent target in \cref{eq:crra-reduced-ol-target}. An independent common-random-number finite-difference audit first checks the corrected HJB reference (\cref{fig:population-audits}). At $131{,}072$ paths, the physical-measure calibration slopes are $1.0154$ for finite difference and $1.0621$ for OL-BPTT; under the tilted Feynman--Kac law they are $1.0146$ and $1.0618$. The tilted relative $L^2$ errors are $1.48\%$ and $6.66\%$, respectively. Tilting raises the median effective-sample-size fraction from $0.645$ to $0.936$ and lowers the top-one-percent weight share from $10.1\%$ to $2.24\%$. Additional paths and a better sampling law therefore reduce uncertainty but do not remove the population fixed-policy defect.

\Cref{cor:olbptt-residual-defect} predicts that the defect decays with the policy-improvement residual. On four exact half-step HJB policy-iteration iterates, using $65{,}536$ tilted paths for each of three seeds, the theorem-covered common regular subset has log--log slope $0.932$ and correlation $0.987$; the all-point aggregation gives $0.938$ and $0.986$. These are descriptive residual-scaled co-decay statistics, not estimates of the occupation-measure constant $\kappa$.

\begin{figure}[!htbp]
\centering
\begin{subfigure}[t]{0.62\textwidth}
  \includegraphics[width=\linewidth]{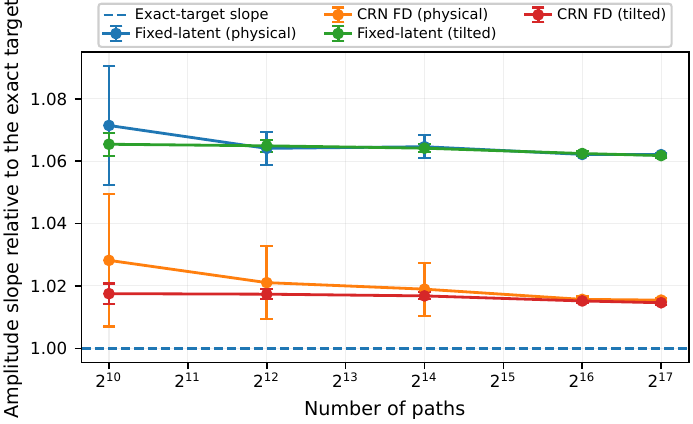}
  \caption{Calibration slope against the corrected HJB target.}
\end{subfigure}\hfill
\begin{subfigure}[t]{0.36\textwidth}
  \includegraphics[width=\linewidth]{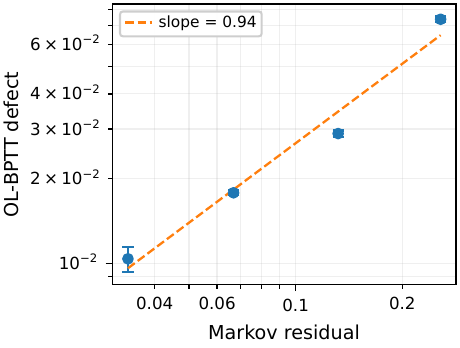}
  \caption{Fixed-latent defect against the HJB residual.}
\end{subfigure}
\caption{Population adjoint--HJB consistency audits. One is ideal in panel~(a); panel~(b) reports all-anchor residual-scaled co-decay along exact damped HJB policy iteration. Error bars show one standard deviation across three seeds.}
\label{fig:population-audits}
\end{figure}

\paragraph{Theorem-matched occupation-measure audit.}
The signed directional ratio
$\widehat\kappa_{\mathrm{dir}}:=\widehat\Xi_{\mathrm A}/\widehat D_{\mathrm A}^{\,2}$
directly targets (A1). The stronger norm-relative ratio
$\widehat\kappa_{\mathrm{occ}}:=\widehat E_{\mathrm A}/\widehat D_{\mathrm A}$
implies (A1) by Cauchy--Schwarz and supplies the sufficient route to (A2) through
$M_{\mathrm A}\le E_{\mathrm A}$. For each of four population OL-BPTT iterates, we form the actual half-step update $u^+=(u+\Acal(u))/2$ and simulate $65{,}536$ paths under its law. The design crosses three starting times, five initial factor states, and three seeds, giving $45$ path banks per iterate and $180$ in total. Both ratios use the same pathwise weighted integrals; per-bank $95\%$ upper endpoints use delta-method covariance estimates after antithetic pairing.
\begin{table}[H]
\centering
\caption{Theorem-matched occupation audit. Each entry reports the maximum point estimate over the banks at that iterate, with the maximum per-bank $95\%$ upper endpoint in parentheses. The half-step threshold is $c_\beta=0.75$.}
\label{tab:kappa-occupation-audit}
\small
\setlength{\tabcolsep}{5pt}
\begin{tabular}{rrcc}
\toprule
Population iterate $k$ & Banks
& \makecell{Max. norm ratio\\$\widehat\kappa_{\mathrm{occ}}$ (95\% U)}
& \makecell{Max. directional ratio\\$\widehat\kappa_{\mathrm{dir}}$ (95\% U)}\\
\midrule
0 & 45 & $0.073807\;(0.073936)$ & $0.066281\;(0.066341)$\\
1 & 45 & $0.065918\;(0.066073)$ & $0.054077\;(0.054178)$\\
2 & 45 & $0.063484\;(0.063667)$ & $0.050064\;(0.050190)$\\
3 & 45 & $0.064889\;(0.065111)$ & $0.048292\;(0.048454)$\\
\midrule
All & 180 & $0.073807\;(0.073936)$ & $0.066281\;(0.066341)$\\
\bottomrule
\end{tabular}
\end{table}
In \cref{tab:kappa-occupation-audit}, every point estimate and upper endpoint is below $0.75$. The largest directional upper endpoint is $0.066341$, while the largest stronger norm-relative upper endpoint is $0.073936$. Thus the direct audit supports directional improvement on the visited banks, and the stronger audit supports the sufficient route to both convergence hypotheses on the same finite design; neither is a proof over the full invariant policy class.

\subsection{Constrained deployment and switching geometry}
\label{sec:constrained-deployment}

In the one-factor box benchmark, current-policy re-evaluation lowers uniform policy RMSE from $3.483\times10^{-2}$ to $2.046\times10^{-3}$ and improves active-set recovery. In the separate 24-iteration output-construction test, averaging the pre-projection field, interpolating continuously, and projecting once reduces exact active-set error from about $4.85\%$ for the rollout-policy average to $0.007\%$ at $N_y=81$ factor grid points. \Cref{fig:box-return} shows that the returned boundary tracks the reference and avoids the staircase created by project-then-interpolate. Richardson, continuation, clamping, and full tabular diagnostics are in Appendix~\ref{app:numerical-details} (including \cref{tab:box-active}).

\begin{figure}[!htbp]
\centering
\includegraphics[width=0.74\textwidth]{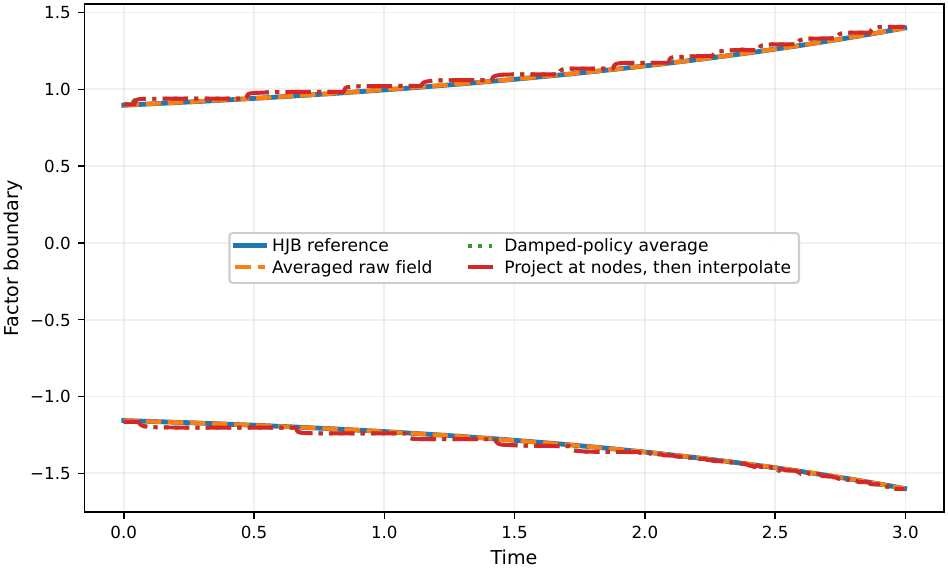}
\caption{Reference and returned active boundaries for box bounds $[-2.5,2.5]$: representative seed $260730$, Richardson with $N_y=81$, 24 outer iterations, and $\beta=1/2$. The returned policy averages the pre-projection field, interpolates continuously, and projects once.}
\label{fig:box-return}
\end{figure}

The constrained three-factor, fifty-asset benchmark uses factor/asset blocks admitting a factorized HJB reference, dense nonlinear cross-factor features, tilted reduced OL-BPTT, and a separable fifty-dimensional QP with heterogeneous box constraints. In the high-precision legacy $1/6$-broad mechanism audit (\cref{fig:flagship}), current-policy and matched pooled on-policy RMSEs are $2.775\times10^{-3}$ and $4.453\times10^{-3}$, a $37.69\%$ reduction. In the separate high-precision $50$--$50$ occupancy/broad-anchor coverage audit, re-evaluation wins under the on-policy and broad laws in all three seeds; pooling wins under the enlarged tail-only law. Appendix~\ref{app:numerical-details} separates these designs.

\begin{figure}[!htbp]
\centering
\includegraphics[width=0.56\textwidth]{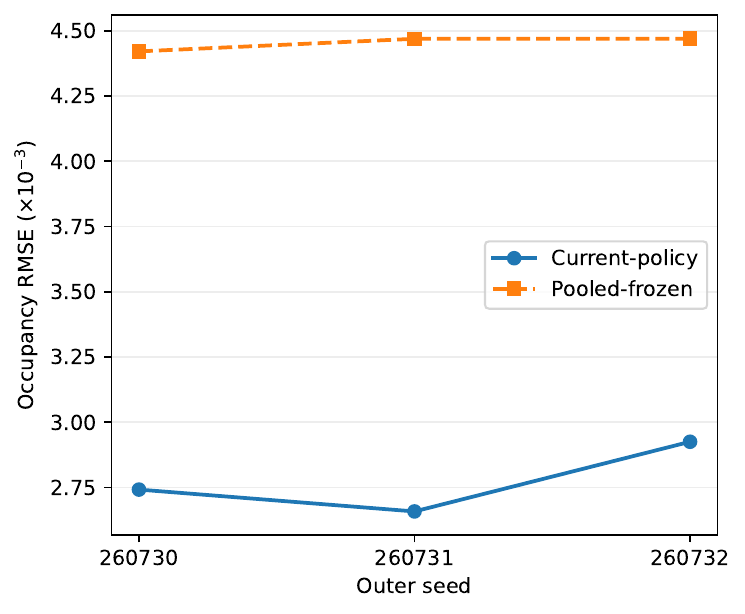}
\caption{Current-policy re-evaluation versus matched pooled refinement: on-policy RMSE in the legacy $1/6$-broad mechanism audit of the three-factor, fifty-asset constrained benchmark.}
\label{fig:flagship}
\end{figure}

Appendix~\ref{app:numerical-details} gives gain and residual profiles (\cref{fig:supp-gain-residual}), three- and five-factor structured regression (\cref{tab:multifactor}), representation-versus-estimation decomposition (\cref{tab:nonlinear}), sampling and representation ablations (\cref{fig:flagship-ablation}), and the design-mixture sweep.

\FloatBarrier
\section{Discussion and conclusion}
\label{sec:discussion}
\label{sec:conclusion}

The learned object is the continuation hedging field rather than an end-to-end many-asset portfolio network. Under CRRA preferences, the constrained action block is decoded by an explicit QP once that field is known. Re-evaluation is useful precisely when the implemented portfolio changes the continuation field: the Merton model is a negative control, while predictable returns create a genuinely policy-dependent improvement operator.

Constraints make switching geometry and coverage part of the numerical problem. Averaging the pre-projection control field and projecting only after interpolation preserves active boundaries better than averaging damped projected policies. On-policy restart points are efficient for the induced occupation law, but broad design points are needed for financially relevant tail states.

The operator hierarchy separates numerical errors from population adjoint--HJB inconsistency. The link $\AcalNum_k\to\Acal$ contains time discretization, conditional estimation, representation, continuation extension, and QP-solver error. The link $\Acal\to\Tcal$ is the population adjoint--HJB consistency problem. Shifted-adjoint cancellation shows that the second PMP adjoint need not approach the value Hessian; on regular future tubes its discrepancy is multiplied by the updated-control diffusion displacement. In the CRRA reduction, the remaining discrepancy is exactly $C(R_{\mathrm{OL}}^u-R^u)$. Global convergence requires two distinct population properties: favorable directional alignment transfers value improvement, while asymptotic stationarity compatibility excludes adjoint fixed points that are not HJB fixed points. The directional audit targets the first property directly, while the audited norm-relative condition is a stronger sufficient route that supplies both.

The limitations are correspondingly sharp. The generic consistency theorem is local and does not cross switching surfaces or verify the global stationarity-compatibility property. The theorem-matched directional and norm audits cover the visited one-factor sequence and stated starts only. Global convergence is proved for the constrained CRRA population subclass, not for a fixed biased finite-sample implementation or for dimension-free factor-state complexity.

We have therefore closed a fixed-latent adjoint-to-control map into a self-consistent constrained policy iteration, supplied a residual-based adjoint--HJB comparison and a true-value certificate, and established global population convergence in the CRRA portfolio subclass. The numerical results show when re-evaluation adds information, when it does not, and how that distinction survives in a three-factor, fifty-asset constrained problem.
\appendix

\section{Proof of the generic adjoint--HJB consistency theorem}
\label{app:supp-generic}

\paragraph{Sufficient local regularity.}
Assumption~\ref{ass:generic-decoded-score-region} follows, for example, if the filtration has the Brownian martingale-representation property; the direct-action coefficients have bounded derivatives through order two on the regular future tube; $V^u\in C^{1,3}$; the covariance is uniformly elliptic on the buffered cylinder; and the fixed-latent first-adjoint defect is a uniformly parabolic system with common scalar principal part and a local $W_p^{2,1}$ estimate for some $p>d_S+2$. For the second adjoint, assume bounded first coefficient derivatives and, for some $q>2$,
\begin{equation}
  \sup_{(t,s)\in K_1}\E_{t,s}^u\left[
  \|\D^2g(S_T^u)\|^q+
  \left(\int_t^T\|\mathcal F_P^u(r,S_r^u)\|^2\dd r\right)^{q/2}
  \right]<\infty.
  \label{eq:second-adjoint-moment-package}
\end{equation}
Finally, on $K$ let $u$, $\Tcal(u)$, and $\Acal(u)$ remain in the relative interior of one regular affine face; assume uniform strong concavity along its tangent space, bounded $b_a$ and $\D_a\sigma$, and local action-Lipschitz continuity of $\sigma$. These conditions are only sufficient and are used below to expose the three estimates entering Theorem~\ref{thm:generic-decoded-score-compatibility}.

\subsection{Fixed-policy value-gradient identification and residual bound for the fixed-latent OL-BPTT defect}
\label{app:compatibility-proof}

As in \cref{sec:compatibility-theory}, the proof is written in direct action coordinates. A moving-chart version would replace primitive state derivatives by derivatives of chart-composed coefficients at fixed latent coordinate, but the residual identity needed below is not claimed for that extension. Fix $(t,s)$ and write $S_r=S_r^{t,s,u}$. Under the assumptions of \cref{lem:closed-loop-compatibility}, the closed-loop coefficients
$b^u(r,s)=b(r,s,u(r,s))$ and $\sigma^u(r,s)=\sigma(r,s,u(r,s))$ are continuously differentiable with bounded derivatives. The stochastic flow is therefore differentiable in its initial state \citep{Kunita1990}. For a coordinate direction $e_i$, define
\begin{equation}
  \Xi_h(r)=\frac{S_r^{t,s+he_i,u}-S_r^{t,s,u}}{h}.
  \label{eq:state-difference-quotient}
\end{equation}
For some $p>1$ chosen large enough for the H\"older conjugates in the stated integrability envelope,
\begin{equation}
  \Xi_h\longrightarrow J_r^u e_i
  \quad\text{in }L^p(\Omega;C([t,T])),
  \label{eq:state-difference-quotient-convergence}
\end{equation}
where
\begin{align}
  \D_s b^u&=b_s+b_a\D_su,\\
  \D_s \sigma^{u,(j)}&=\sigma_s^{(j)}+\sigma_a^{(j)}\D_su.
\end{align}

Write $\Delta_hS_r=S_r^{t,s+he_i,u}-S_r^{t,s,u}$. The terminal difference quotient has the mean-value representation
\begin{equation}
  \frac{g(S_T+\Delta_hS_T)-g(S_T)}{h}
  =\int_0^1
  \D g(S_T+\lambda\Delta_hS_T)^\top \Xi_h(T)\dd\lambda.
  \label{eq:terminal-mean-value}
\end{equation}
The terminal family in \cref{eq:terminal-mean-value} is uniformly integrable by the envelope in \cref{lem:closed-loop-compatibility}. In the usual polynomial-growth case this follows from H\"older's inequality, positive-moment bounds for the nearby flows, and the $L^p$ bound for $\Xi_h$. For full-wealth CRRA utility, \cref{eq:crra-terminal-envelope,eq:crra-terminal-ui} give the corresponding negative-moment and normalized-tangent verification. In the factor-reduced formulation the terminal datum is constant. Equation \eqref{eq:state-difference-quotient-convergence}, continuity of $\D g$, and Vitali's theorem therefore give convergence in $L^1$ to $\D g(S_T)^\top J_T^ue_i$.

The running term is handled similarly. Since
\begin{equation}
  \D_s\ell^u=\ell_s+(\D_su)^\top\ell_a
  \label{eq:closed-loop-running-derivative}
\end{equation}
 is bounded under the assumptions of Lemma~\ref{lem:closed-loop-compatibility}, the running-payoff difference quotients are uniformly integrable on $[t,T]\times\Omega$. Fubini's theorem and the same $L^1$ argument yield
\begin{equation}
  \frac{\dd}{\dd s_i}
  \E_{t,s}\left[\int_t^T\ell^u(r,S_r)\dd r\right]
  =\E_{t,s}\left[\int_t^T
  (J_r^ue_i)^\top\D_s\ell^u(r,S_r)\dd r\right].
\end{equation}
Combining the terminal and running terms for every coordinate $i$ proves \cref{eq:closed-loop-value-gradient}.

For the fixed-latent OL-BPTT comparison, let $J_{\mathrm{FL}}$ solve the variational equation with coefficient matrices $b_s$ and $\sigma_s^{(j)}$ only, all evaluated along the same fixed-policy path. Replacing $J^u$ by $J_{\mathrm{FL}}^u$ and the total running derivative by the partial derivative $\ell_s$ gives \cref{eq:fixed-latent-target}. Subtracting the two payoff-derivative representations gives the exact decomposition \cref{eq:fixed-latent-defect}.

\paragraph{Envelope representation and residual scaling.}
Under the additional classical regularity of \cref{cor:olbptt-residual-defect}, let $p^u:=\D_sV^u$ and let $z$ be a vector field. Write $\Lcal^u z$ componentwise and define the fixed-latent first-variation operator
\begin{equation}
\begin{split}
  \mathscr M_{\mathrm{FL}}^u z
  :={}&\partial_t z+\Lcal^u z+(b_s)^\top z\\
  &+\sum_{j=1}^{d_W}
  (\sigma_s^{(j)})^\top(\D_sz)\sigma^{u,(j)}.
  \label{eq:fixed-latent-vector-operator}
\end{split}
\end{equation}
Differentiating the fixed-policy equation in the state gives
\begin{equation}
  \mathscr M_{\mathrm{FL}}^u p^u+\ell_s
  +(\D_su)^\top\left[
    \ell_a+b_a^\top p^u
    +\sum_{j=1}^{d_W}
    (\sigma_a^{(j)})^\top(\D_sp^u)\sigma^{u,(j)}
  \right]=0.
  \label{eq:markov-gradient-vector-equation}
\end{equation}
The bracket is exactly $\nabla_a\Qcal^u(t,s,u(t,s))$. By the vector Feynman--Kac formula associated with the multiplicative tangent $J_{\mathrm{FL}}^u$, the OL-BPTT target in \cref{eq:fixed-latent-target} solves
\begin{equation}
  \mathscr M_{\mathrm{FL}}^u G_{\mathrm{OL}}^u+\ell_s=0,
  \qquad G_{\mathrm{OL}}^u(T,\cdot)=\D g.
  \label{eq:olbptt-target-vector-equation}
\end{equation}
Therefore $\delta^u:=p^u-G_{\mathrm{OL}}^u$ satisfies
\begin{equation}
  \mathscr M_{\mathrm{FL}}^u\delta^u+r_{\mathrm{env}}^u=0,
  \qquad \delta^u(T,\cdot)=0.
  \label{eq:olbptt-defect-vector-equation}
\end{equation}
Applying the same vector Feynman--Kac formula proves \cref{eq:olbptt-envelope-representation}. Under \cref{eq:weighted-quadratic-advantage,eq:quadratic-score-general}, action differentiation gives $\nabla_a\Qcal^u=w^u\nabla_aq^u$. The KKT normal--tangent cancellation stated in \cref{cor:olbptt-residual-defect} then yields \cref{eq:olbptt-envelope-residual-identity}; the operator-norm estimate yields \cref{eq:olbptt-residual-scaled-bound}.

\paragraph{Full-wealth CRRA terminal envelope.}
For $U(x)=x^{1-\gamma}/(1-\gamma)$, $\gamma>1$, let $X_{T,h}^\lambda$ be the positive mean-value segment between nearby terminal wealths and let $\Xi_h^X(T)$ be the wealth component of the state-flow difference quotient. Then
\begin{equation}
  |U'(X_{T,h}^\lambda)\Xi_h^X(T)|
  =(X_{T,h}^\lambda)^{1-\gamma}
  \left|\frac{\Xi_h^X(T)}{X_{T,h}^\lambda}\right|.
  \label{eq:crra-terminal-envelope}
\end{equation}
A sufficient uniform-integrability condition is, for some $\delta>0$,
\begin{equation}
  \sup_{0<|h|\le h_0}\sup_{0\le\lambda\le1}
  \E\left[\left((X_{T,h}^\lambda)^{1-\gamma}
  \left|\frac{\Xi_h^X(T)}{X_{T,h}^\lambda}\right|\right)^{1+\delta}\right]<\infty.
  \label{eq:crra-terminal-ui}
\end{equation}
This follows from sufficiently high negative moments of nearby wealth paths and positive moments of the normalized tangent. In the factor-reduced CRRA problem the terminal datum is the constant one, so this full-wealth envelope is not needed.

\subsection{Complete proof of Theorem~\ref{thm:generic-decoded-score-compatibility}}
\label{app:decoded-score-compatibility-proof}
Let $G:=G_{\mathrm{OL}}^u$, $p^u:=\D_sV^u$, $\Gamma:=\D_{ss}^2V^u$, and $\delta:=p^u-G$. Applying It\^o--Krylov to the vector equation above and using uniqueness of the linear BSDE \cref{eq:first-adjoint-bsde} yields the Markovian identification \citep{PardouxPeng1992}:
\begin{equation}
  \lambda^u=G(t,S_t^u),\qquad
  \mathsf Z^u=(\D_sG)\sigma^u,
  \qquad
  \mathscr M_{\mathrm{FL}}^u\delta+r_{\mathrm{env}}^u=0,
  \quad \delta(T,\cdot)=0.
  \label{eq:generic-Markov-BSDE-identification-proof}
\end{equation}

\emph{Step 1: first-adjoint level and martingale coefficient.}
The envelope representation and the regular-tube residual bound imply on the buffered cylinder
$\|\delta\|_\infty+\|r_{\mathrm{env}}^u\|_\infty\le C\mathfrak r_{\mathcal D}(u)$. Applying the componentwise interior estimate to the parabolic system, with the lower-order gradient coupling absorbed by interpolation, yields
\begin{equation}
  \|\delta\|_{W_p^{2,1}(K)}\le C\bigl(\|\delta\|_{L^p(K_1)}+\|r_{\mathrm{env}}^u\|_{L^p(K_1)}\bigr),
  \qquad
  \|\D_s\delta\|_{L^\infty(K)}\le C\mathfrak r_{\mathcal D}(u).
  \label{eq:generic-vector-Wp-estimate}
\end{equation}
Since $\mathsf Z^u-\Gamma\sigma^u=-(\D_s\delta)\sigma^u$, boundedness of $\sigma^u$ gives
\begin{equation}
  \|\lambda^u-p^u\|_{L^\infty(K)}+
  \|\mathsf Z^u-\Gamma\sigma^u\|_{L^\infty(K)}
  \le C_Z\mathfrak r_{\mathcal D}(u).
  \label{eq:generic-lambda-Z-residual}
\end{equation}

\emph{Step 2: second-adjoint bound.}
The restarted matrix adjoint is the linear BSDE in \cref{eq:second-adjoint-bsde}, with coefficients and source defined in \cref{eq:adjoint-coefficient-definitions,eq:second-adjoint-driver-definition}. It\^o's formula for $e^{\alpha r}\|P_r^u\|_F^2$, Young's inequality, and \cref{eq:second-adjoint-moment-package} give a deterministic restart bound
\begin{equation}
  \|P^u\|_{L^\infty(K)}+\|P^u-\Gamma\|_{L^\infty(K)}\le B_{P\Gamma,K}<\infty,
  \label{eq:generic-P-Gamma-local-bound}
\end{equation}
with no smallness claim in the policy residual.

\emph{Step 3: candidate displacement.}
Let $\Pi_L$ project onto the tangent space of the common face. At $a=u$, Proposition~\ref{prop:shifted-gradient-cancellation} and \cref{eq:generic-lambda-Z-residual} give
$\|\Pi_L[\nabla_a\Hgen_{\mathrm A}^u(u)-\nabla_a\Qcal^u(u)]\|\le C\mathfrak r_{\mathcal D}(u)$.
Tangential KKT stationarity at $\Tcal(u)$ and Lipschitz continuity of the HJB gradient give
$\|\Pi_L\nabla_a\Qcal^u(u)\|\le L\mathfrak r_{\mathcal D}(u)$.
Strong concavity of the adjoint Hamiltonian and stationarity at $\Acal(u)$ therefore imply
\begin{equation}
  \|\Acal(u)-u\|_{L^\infty(K)}\le C_{\mathrm{step}}\mathfrak r_{\mathcal D}(u).
  \label{eq:generic-candidate-displacement}
\end{equation}

\emph{Step 4: Hamiltonian and operator discrepancy.}
Let $L_\sigma$ be the local action-Lipschitz constant of $\sigma$. Evaluate \cref{eq:shifted-gradient-cancellation} at $a=\Acal(u)$. The first-adjoint blocks are controlled by \cref{eq:generic-lambda-Z-residual}; for the second-adjoint block,
\[
  \|(P^u-\Gamma)[\sigma(\Acal(u))-\sigma(u)]\|_{L^\infty(K)}
  \le B_{P\Gamma,K}L_\sigma C_{\mathrm{step}}\mathfrak r_{\mathcal D}(u).
\]
This proves \cref{eq:generic-decoded-score-residual}. Finally, tangential stationarity at $\Acal(u)$ and $\Tcal(u)$ and $\mu_{\mathrm M}$-strong concavity of the HJB action Hamiltonian give
\[
  \mu_{\mathrm M}\|\Acal(u)-\Tcal(u)\|
  \le \|\nabla_a\Qcal^u(\Acal(u))-\nabla_a\Hgen_{\mathrm A}^u(\Acal(u))\|,
\]
which proves \cref{eq:generic-map-distance-residual}. If the adjoint and HJB Hamiltonians are divided by a common positive action-independent normalization, the same proof applies without a normalizer-mismatch term, as in the CRRA reduction.

\section{Weighted comparison, verification, and global CRRA convergence}
\label{app:crra-global-proof}

\subsection{Weighted comparison and terminal attachment}
\label{app:bellman-identification}
Throughout this subsection, $A_{\max}:=\sup_{a\in\PortSet}|a|$ and $M_Z\succ0$ is a Lyapunov matrix for the positive stable $K_Z$:
\begin{equation}
  K_Z^\top M_Z+M_ZK_Z\succeq2c_ZI,\qquad c_Z>0.
  \label{eq:lyapunov-matrix}
\end{equation}
Define
\begin{equation}
  \rho_Z(z):=(1+z^\top M_Zz)^{1/2},\qquad
  \Psi_{\hat\eta}(z):=e^{\hat\eta\rho_Z(z)},\qquad
  \Lcal_Z^a:=\bigl[K_Z(\bar z-z)+\chi C^\top a\bigr]^\top\D_z
  +\tfrac12\tr(\Lambda\Lambda^\top\D_{zz}^2).
  \label{eq:barrier-weight}
\end{equation}
Elementary differentiation gives
\begin{equation}
  \D\Psi_{\hat\eta}=\hat\eta\Psi_{\hat\eta}\D\rho_Z,
  \qquad
  \D^2\Psi_{\hat\eta}=\Psi_{\hat\eta}
  \bigl(\hat\eta\D^2\rho_Z+\hat\eta^2\D\rho_Z\otimes\D\rho_Z\bigr).
  \label{eq:barrier-derivatives}
\end{equation}

\begin{lemma}[Lyapunov barrier supersolution]
\label{lem:barrier-supersolution}
There is $c_\flat>0$, depending only on $K_Z$ and $M_Z$, such that for every $\hat\eta>\eta_*:=|\chi|C_1/c_\flat$ some finite $\lambda(\hat\eta)$ satisfies
\begin{equation}
  \sup_{a\in\PortSet}\left[
  \Psi_{\hat\eta}^{-1}\Lcal_Z^a\Psi_{\hat\eta}(z)+\chi g(z,a)
  \right]\le\lambda(\hat\eta),\qquad z\in\R^{d_z}.
  \label{eq:barrier-inequality}
\end{equation}
\end{lemma}

\begin{proof}
The diffusion and bounded-control terms satisfy, uniformly in $z$ and $a$,
\begin{align*}
  \tfrac12\Psi^{-1}\tr(\Lambda\Lambda^\top\D^2\Psi)
  &\le \|\Lambda\|_F^2\bigl(\hat\eta\|M_Z\|+\tfrac12\hat\eta^2\|M_Z\|\bigr),\\
  \Psi^{-1}(\chi C^\top a)^\top\D\Psi
  &\le |\chi|A_{\max}\|C\|\hat\eta\|M_Z\|^{1/2}.
\end{align*}
Moreover,
$K_Z(\bar z-z)^\top M_Zz\le C|z|-c_Z|z|^2$. Since
$\rho_Z(z)\le1+\|M_Z\|^{1/2}|z|$, there is
$c_\flat=c_Z/[2(1+\|M_Z\|^{1/2})]>0$ such that, for $|z|\ge1$,
\begin{equation}
  \Psi^{-1}K_Z(\bar z-z)^\top\D\Psi
  \le\hat\eta(C_{\mathrm{ou}}-c_\flat|z|).
  \label{eq:ou-dissipation}
\end{equation}
Finally, \cref{eq:hamiltonian-growth} gives
$\chi g(z,a)\le|\chi|(C_0+C_1|z|)$. Hence the coefficient of $|z|$ is
$|\chi|C_1-\hat\eta c_\flat<0$ when $\hat\eta>\eta_*$, and the expression is bounded on the remaining compact region.
\end{proof}

\begin{proof}[Proof of Proposition~\ref{prop:crra-analytic-closure}]
\emph{Weighted comparison.}
Let $W=F_1-F_2$, with common terminal value zero, and select the a.e.-defined measurable maximizer
\[
  \widehat a(t,z)=\mathcal G[F_2](t,z)
  =\Pi_{\PortSet}^\Sigma\!\left(\frac1\gamma\Sigma^{-1}
  [\mu(z)+C\D_z\log F_2(t,z)]\right).
\]
It is bounded because it is $\PortSet$-valued. For $i=1,2$, set $\phi_i(a):=g(z,a)F_i(t,z)+a^\top C\D_zF_i(t,z)$. At the maximizer of the second Hamiltonian,
$\sup_a\phi_1(a)-\sup_a\phi_2(a)\ge\phi_1(\widehat a)-\phi_2(\widehat a)$; because $\chi<0$, this gives
\begin{equation}
  \Hgen_{\mathrm B}(z,F_1,\D_zF_1)
  -\Hgen_{\mathrm B}(z,F_2,\D_zF_2)
  \le\chi\left[g(z,\widehat a)W
  +\widehat a^{\,\top}C\D_zW\right]
  \qquad\text{a.e.}
  \label{eq:selector-linearization}
\end{equation}
Subtracting the two HJB equations therefore gives the strong linear inequality
\begin{equation}
  \partial_tW+\Lcal_Z^{\widehat a}W+\widetilde cW\ge0,
  \qquad
  \widetilde c(t,z):=\chi g(z,\widehat a(t,z)),
  \quad |\widetilde c|\le|\chi|(C_0+C_1|z|).
  \label{eq:linearized-inequality}
\end{equation}
Choose
$\hat\eta>\max\{\eta_*,\eta_0/\lambda_{\min}(M_Z)^{1/2}\}$ and
$\lambda\ge\lambda(\hat\eta)$, and put
\begin{equation}
  \varphi_\varepsilon=W-\varepsilon e^{\lambda(T-t)}\Psi_{\hat\eta}.
  \label{eq:comparison-barrier}
\end{equation}
The exponent separation makes $\varphi_\varepsilon\to-\infty$ uniformly as $|z|\to\infty$. The barrier lemma also gives
\[
  \partial_t(e^{\lambda(T-t)}\Psi)+\Lcal_Z^{\widehat a}(e^{\lambda(T-t)}\Psi)
  +\widetilde c e^{\lambda(T-t)}\Psi\le0,
\]
so $\varphi_\varepsilon$ satisfies \cref{eq:linearized-inequality} and is negative at $t=T$.

If $\sup\varphi_\varepsilon>0$, spatial decay localizes its positive maximum to a bounded cylinder $[0,T]\times B_R$, where $B_R:=\{z\in\R^{d_z}:|z|<R\}$. Since $\varphi_\varepsilon$ is continuous and strictly negative at $t=T$, some $\delta>0$ satisfies $\varphi_\varepsilon\le0$ on $[T-\delta,T]\times B_R$. It therefore suffices to work on $[0,T-\delta]\times B_R$. Reverse time, $s=T-\delta-t$, and write
$\psi(s,z)=\varphi_\varepsilon(T-\delta-s,z)$. After the substitution
$\psi=e^{\Lambda_Rs}\widetilde\psi$, with
$\Lambda_R\ge\sup\widetilde c^+$ on the cylinder, the parabolic
Alexandrov--Bakelman--Pucci--Krylov maximum principle for
$W_p^{2,1}$ strong solutions, $p>d_z+2$ \citep{Krylov1987,Lieberman1996}, contradicts a positive interior maximum. Thus
$\varphi_\varepsilon\le0$; letting $\varepsilon\downarrow0$ gives
$F_1\le F_2$, and interchanging them gives equality.

\emph{Uniform terminal attachment.}
The fixed-policy Feynman--Kac representation is
\begin{equation}
  F^u(t,z)=\E_{t,z}^{Q_u}\left[
  \exp\left\{\chi\int_t^Tg(Z_s,u_s)\dd s\right\}\right],
  \label{eq:fixed-policy-fk-global}
\end{equation}
where under $Q_u$ the factor drift is
$K_Z(\bar z-Z)+\chi C^\top u$. Variation of constants writes $Z$ as a uniformly bounded pathwise drift term plus a Gaussian stochastic convolution. The Dambis--Dubins--Schwarz representation and the Brownian reflection bound therefore imply, for every finite $c$ and $R$,
\begin{equation}
  \sup_{u:\,[0,T]\times\R^{d_z}\to\PortSet\ \mathrm{Borel}}
  \sup_{|z|\le R}\E_{t,z}^{Q_u}\exp\left\{c\sup_{t\le s\le T}|Z_s|\right\}<\infty.
  \label{eq:uniform-exponential-moments-fixed}
\end{equation}
If $A=\chi\int_t^Tg(Z_s,u_s)\dd s$, then
$|A|\le|\chi|(T-t)(C_0+C_1\sup|Z_s|)$. The inequality
$|e^A-1|\le|A|e^{|A|}$ and \cref{eq:uniform-exponential-moments-fixed} yield
$|F^u(t,z)-1|\le C_R(T-t)$ uniformly over bounded Borel Markov $u$ and $|z|\le R$, proving \cref{eq:terminal-attachment}.
\end{proof}

\subsection{Exact policy iteration and reduced verification}
\label{app:crra-exact-global-proof}
\begin{proof}[Completion of the proof of Theorem~\ref{thm:crra-global-convergence}]
\emph{Monotone limit and compactness.}
Exact value improvement and $V^u=x^\chi F^u/\chi<0$ give
$V^{u^k}\uparrow\overline V\le0$ and
$F^{u^k}\downarrow\overline F>0$. The policy-uniform
$W_p^{2,1}$ estimates in Proposition~\ref{prop:crra-fixed-policy-regularity}, together with the compact embedding \cref{eq:crra-compact-embedding}, give local precompactness. Monotone pointwise convergence identifies the function component of every convergent subsequence with the same limit; hence both full sequences $F^{u^k}$ and $\D_zF^{u^k}$ converge locally in the stated parabolic H\"older topology, and
$\mathcal G[F^{u^k}]\to\mathcal G[\overline F]$ locally uniformly.

\emph{Residual decay.}
The greedy fields and damped policies have a common local modulus. Telescoping the exact improvement bound makes the weighted occupation norm of
$\Tcal(u^k)-u^k$ summable. If the residual failed to vanish uniformly on a compact cylinder, equicontinuity would give a common space--time neighborhood where its squared norm exceeds $\varepsilon^2/4$. The short-time occupation bound \cref{eq:crra-scale-coverage} would then force a fixed positive one-step value increment at the corresponding moving start, contradicting locally uniform convergence of the monotone values. Hence
\begin{equation}
  \|\Tcal(u^k)-u^k\|_{L^\infty(K_{\tau,R})}\to0,
  \qquad u^k\to\mathcal G[\overline F]
  \quad\text{locally uniformly}.
  \label{eq:crra-exact-proof-residual}
\end{equation}
Strong convergence of $F^{u^k},\D_zF^{u^k}$, weak local compactness of the second spatial and first time derivatives, and policy convergence permit passage to \cref{eq:crra-global-fixed-policy}; thus $\overline F$ solves \cref{eq:crra-global-hjb}. Proposition~\ref{prop:crra-analytic-closure} attaches the terminal condition and makes this solution unique.

\emph{Reduction and integrability.}
For any progressively measurable $\PortSet$-valued control $\nu$, It\^o's formula for $\log X^\nu$ and the CRRA exponential martingale give
\begin{equation}
  \E_{t,x,z}[U(X_T^\nu)]
  =\frac{x^\chi}{\chi}\E_{t,z}^{Q_\nu}
  \exp\left\{\chi\int_t^Tg(Z_s,\nu_s)\dd s\right\}
  =:\frac{x^\chi}{\chi}F^\nu(t,z),
  \label{eq:reduced-control-representation}
\end{equation}
where under $Q_\nu$ the factor drift is
$K_Z(\bar z-Z)+\chi C^\top\nu$. Boundedness of $\nu$ makes Novikov immediate, and variation of constants yields
\begin{equation}
  \sup_\nu\E_{t,z}^{Q_\nu}\exp\left\{c\sup_{t\le s\le T}|Z_s|\right\}<\infty
  \qquad(c<\infty).
  \label{eq:uniform-exponential-moments}
\end{equation}

\emph{It\^o--Krylov verification.}
Let
$Y_s=\overline F(s,Z_s)\exp\{\chi\int_t^sg(Z_r,\nu_r)\dd r\}$ and, for $n$ sufficiently large that $T-1/n>t$, set
\[
  \tau_n:=\inf\{s\ge t:\ |Z_s|\ge n\}\wedge(T-1/n).
\]
Because $\overline F\in W_{p,\mathrm{loc}}^{2,1}([0,T)\times\R^{d_z})$, It\^o--Krylov applies up to $s\wedge\tau_n$, and
\[
  \dd Y_s=e^{\chi\int_t^sg\dd r}
  [\partial_t\overline F+\Lcal_Z^{\nu_s}\overline F+\chi g(Z_s,\nu_s)\overline F]\dd s
  +\dd(\text{local martingale}).
\]
Since $\chi<0$, the HJB equation makes the drift nonnegative. Because
$\overline F\in\mathcal C_{\eta_0}$, some $C_{\overline F}<\infty$ satisfies
\begin{equation}
  |Y_{s\wedge\tau_n}|
  \le C_{\overline F}\exp\!\left\{
  \eta_0\langle Z_{s\wedge\tau_n}\rangle
  +|\chi|\int_t^{s\wedge\tau_n}(C_0+C_1|Z_r|)\dd r
  \right\}.
  \label{eq:verification-ui-envelope}
\end{equation}
The exponential-moment bound \cref{eq:uniform-exponential-moments} therefore makes the stopped family uniformly integrable. Letting $n\to\infty$, using continuity of $\overline F$ on $[0,T]\times\R^{d_z}$ and the terminal condition, removes both the spatial and terminal stopping and gives
$\overline F(t,z)\le F^\nu(t,z)$. Multiplication by
$x^\chi/\chi<0$ proves the verification inequality.

For $u^\star=\mathcal G[\overline F]$, the pointwise supremum is attained and the drift vanishes. Its factor dynamics are a bounded drift perturbation of the Gaussian OU law. With bounded kernel
$\theta(t,z)=\Lambda^\top(\Lambda\Lambda^\top)^{-1}\chi C^\top u^\star(t,z)$,
Girsanov gives weak existence, and removing the same kernel from any weak solution gives uniqueness in law \citep{KaratzasShreve1991}. Equality follows, so $\overline F=F^\star$, $u^\star$ is optimal, and the theorem follows.
\end{proof}

\subsection{Population OL-BPTT policy iteration}
\label{app:crra-adjoint-global-proof}
\begin{proof}[Completion of the proof of Theorem~\ref{thm:crra-adjoint-global}]
Write
$F_k=F^{u^k}$, $G_k=\Tcal(u^k)=\mathcal G[F_k]$, and
$A_k=\Acal(u^k)$. For a fixed starting state let
$D_k,E_k,R_k,M_k$ denote the weighted path--time quantities in
\cref{eq:occupancy-adjoint-movement,eq:occupancy-adjoint-defect,eq:occupancy-markov-residual,eq:occupancy-map-distance}. The directional condition and Corollary~\ref{cor:occupancy-relative-adjoint-step} give
\begin{equation}
  V^{u^{k+1}}-V^{u^k}
  \ge c_{\mathrm A}D_k^2,
  \qquad
  c_{\mathrm A}=\underline\beta(c_{\overline\beta}-\kappa_{\mathrm{dir}})>0.
  \label{eq:crra-adjoint-proof-directional-increment}
\end{equation}
Thus the values increase and $\sum_kD_k^2<\infty$ for every fixed start. Since $\chi<0$, the factors $F_k$ decrease. Proposition~\ref{prop:crra-fixed-policy-regularity}, positivity, and the common local modulus imply that $F_k$ converges locally uniformly to a positive $\overline F$. Hence the consecutive value increments vanish locally uniformly. Applying \cref{eq:crra-adjoint-proof-directional-increment} with initial wealth one yields, for every compact cylinder $K_{\tau,R}$,
\begin{equation}
  \sup_{(t,z)\in K_{\tau,R}}
  D_{\mathrm A}^{u^k,\beta_k}(t,1,z)\longrightarrow0.
  \label{eq:crra-adjoint-proof-movement-starts}
\end{equation}
The moving-start condition \cref{eq:crra-stationarity-compatibility} then implies
\begin{equation}
  \sup_{(t,z)\in K_{\tau,R}}
  M_{\mathrm A}^{u^k,\beta_k}(t,1,z)\longrightarrow0.
  \label{eq:crra-adjoint-proof-map-starts}
\end{equation}
Indeed, otherwise a subsequence of moving starts with $M_k$ bounded away from zero would contradict \cref{eq:crra-adjoint-proof-movement-starts,eq:crra-stationarity-compatibility}.

Fixed-policy compactness also yields
\begin{equation}
  G_k=\mathcal G[F_k]\to\mathcal G[\overline F]
  \quad\text{locally uniformly}.
  \label{eq:crra-adjoint-proof-markov-limit}
\end{equation}
Proposition~\ref{prop:crra-fixed-policy-regularity} and the damped recursion give common local moduli for $A_k-u^k$ and $A_k-G_k$. Suppose $A_k-u^k$ failed to vanish on $K_{\tau,R}$. There would be $\varepsilon>0$ and moving points $(t_k,z_k)$ at which its $\gamma\Sigma$-norm is at least $\varepsilon$. Equicontinuity gives fixed radii and a short time window on which the squared movement remains at least $\varepsilon^2/4$. Applying \cref{eq:crra-scale-coverage} under the actual updated-policy law gives a fixed positive lower bound for the corresponding $D_k^2$, contradicting \cref{eq:crra-adjoint-proof-movement-starts}. Therefore
\begin{equation}
  \|A_k-u^k\|_{L^\infty(K_{\tau,R})}\to0.
  \label{eq:crra-adjoint-proof-adjoint-movement}
\end{equation}
The identical moving-restart argument, now using \cref{eq:crra-adjoint-proof-map-starts}, gives
\begin{equation}
  \|A_k-G_k\|_{L^\infty(K_{\tau,R})}\to0.
  \label{eq:crra-adjoint-proof-map-distance}
\end{equation}
Combining \cref{eq:crra-adjoint-proof-adjoint-movement,eq:crra-adjoint-proof-map-distance} with the triangle inequality yields
\begin{equation}
  \|G_k-u^k\|_{L^\infty(K_{\tau,R})}\to0.
  \label{eq:crra-adjoint-proof-markov-residual}
\end{equation}

Combining \cref{eq:crra-adjoint-proof-markov-limit,eq:crra-adjoint-proof-markov-residual} gives
$u^k\to\overline u:=\mathcal G[\overline F]$. Proposition~\ref{prop:crra-analytic-closure} attaches the terminal datum. Strong convergence of the factors and first derivatives, weak compactness of the remaining derivatives, and policy convergence permit passage to the fixed-policy equation, so $\overline F$ is a strong reduced HJB solution. Weighted comparison and the verification above identify
$(\overline F,\overline u)=(F^\star,u^\star)$, completing the proof.
\end{proof}

\section{Supplementary numerical diagnostics}
\label{app:numerical-details}

\subsection{Gain and residual-profile diagnostics}

A deterministic fixed-policy PDE line profile along the reconstructed second-update direction places the best update near one half (\cref{fig:gain-profile}). This motivates the fixed gain $\beta=1/2$ in the damped experiments; the smooth one-factor diagnostic in Section~\ref{sec:when-iterate} instead retains its sampled log-CE acceptance test. The exact local HJB Hamiltonian favors a full step, whereas deterministic field and finite QP-tolerance error can produce an interior maximizer of the approximate certificate. The normalized residual decay in \cref{fig:codecay-normalized} complements the log--log panel in \cref{fig:population-audits}.

\begin{figure}[!htbp]
\centering
\begin{subfigure}[t]{0.49\textwidth}
  \includegraphics[width=\linewidth]{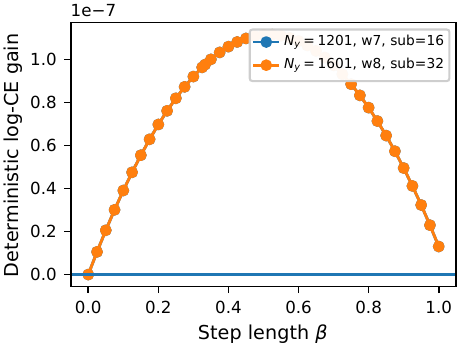}
  \caption{Deterministic line profile along the second-update direction.}
  \label{fig:gain-profile}
\end{subfigure}\hfill
\begin{subfigure}[t]{0.49\textwidth}
  \includegraphics[width=\linewidth]{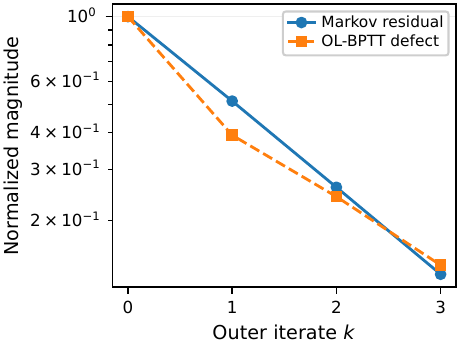}
  \caption{Normalized all-anchor defect and residual.}
  \label{fig:codecay-normalized}
\end{subfigure}
\caption{Additional gain and residual-profile diagnostics. In panel~(a), w7/w8 denote factor-domain half-widths of seven/eight stationary standard deviations, and sub=16/32 denote PDE substeps per policy interval; the grids have $N_y=1201/1601$.}
\label{fig:supp-gain-residual}
\end{figure}

\subsection{One-factor numerical integrity and output construction}

The wealth-curvature ratio is correct to machine precision, while the mixed continuation ratio has first-order time bias. With $N_t$ time intervals, coupled Richardson reduces the $N_t=128$ operator bias from about $1.30\times10^{-4}$ to $2.36\times10^{-5}$ without measurable variance inflation. Cubic time interpolation, exact terminal closure, and structure-compatible tail extrapolation remove a false fixed point caused by factor clamping; under the same protocol, clamping raises uniform policy RMSE from $2.05\times10^{-3}$ to $7.35\times10^{-2}$. With those corrections, the smooth benchmark reaches native-grid RMSE $1.24\times10^{-5}$ after 40 outer iterations with $\beta=1/2$ and adaptive post-plateau averaging.

\begin{table}[!htbp]
\centering
\caption{One-factor box-constrained updates. First block: separate 40-iteration runs with $\beta=1/2$. The RMSE row compares the all-update average of feasible frozen-policy updates with the adaptive post-plateau average of damped current-policy iterates; the active-set and boundary rows compare the one-shot update with that output. Second block: separate 24-iteration trajectories at $N_y=81$, with post-plateau raw-field averaging, interpolation, and final projection; the columns compare standard and Richardson-corrected post-processing. Active-set errors without \% are fractions.}
\label{tab:box-active}
\footnotesize
\setlength{\tabcolsep}{3pt}
\begin{tabular}{@{}>{\raggedright\arraybackslash}p{0.58\textwidth}rr@{}}
\toprule
Metric & Baseline & Corrected \ipgdpo\\
\midrule
\multicolumn{3}{@{}l}{\emph{Iterative mechanism: frozen/one-shot baseline versus current-policy update}}\\
Uniform policy RMSE & $3.483\times10^{-2}$ & $2.046\times10^{-3}$\\
Active-set error & $1.508\%$ & $0.730\%$\\
Boundary RMSE sum & $1.568\times10^{-1}$ & $8.051\times10^{-2}$\\
\midrule
\multicolumn{3}{@{}l}{\emph{Output-policy construction: standard versus Richardson post-processing}}\\
Uniform RMSE & $1.856\times10^{-3}$ & $1.868\times10^{-4}$\\
On-policy RMSE & $1.962\times10^{-3}$ & $2.835\times10^{-4}$\\
Exact active-set error & $6.35\times10^{-4}$ & $6.99\times10^{-5}$\\
Boundary RMSE sum & $6.10\times10^{-3}$ & $9.49\times10^{-4}$\\
\bottomrule
\end{tabular}
\end{table}

\subsection{Multifactor scaling and representation regimes}

The unconstrained three- and five-factor tests use a matrix-Riccati reference, retained under the nonlinear observed-coordinate transformation. Both \cref{tab:multifactor,tab:nonlinear} use $\beta=1/2$ and post-plateau averaging of pre-projection fields. In \cref{tab:multifactor}, the field is represented at nine time knots by affine regression, followed by a fifty-dimensional QP; high-precision and matched-cost runs use 12 and 8 outer iterations, respectively. Re-evaluation lowers one-shot error and beats matched frozen-policy averaging. Richardson is preferable when time bias dominates; single-grid $N_t=64$ OL-BPTT is more efficient when multifactor regression noise dominates.

\begin{table}[!htbp]
\centering
\caption{Non-tabular structured-regression \ipgdpo with fifty assets. Errors are policy RMSEs under the on-policy law; coefficients count fitted adjoint-field coefficients.}
\label{tab:multifactor}
\small
\begin{tabular}{rrrrrr}
\toprule
Factors & \makecell{Adjoint-field\\coefficients} & \makecell{High-precision\\one-shot} & \makecell{High-precision\\returned} & \makecell{Matched-cost\\current} & \makecell{Matched-cost\\frozen}\\
\midrule
3 & $108$ & $4.347\times10^{-3}$ & $1.208\times10^{-3}$ & $1.636\times10^{-3}$ & $4.456\times10^{-3}$\\
5 & $270$ & $7.433\times10^{-3}$ & $1.846\times10^{-3}$ & $2.374\times10^{-3}$ & $6.167\times10^{-3}$\\
\bottomrule
\end{tabular}
\end{table}

In \cref{tab:nonlinear}, latent OU factors are observed through a smooth nonlinear coordinate transform, using 12 outer iterations. A noiseless-target ridge fit on a fixed 512-anchor bank gives the oracle-fit benchmark under the same features and regularization. Quadratic regression returns near this error, so iteration cannot repair its dominant misspecification. Fixed tanh features reduce the oracle-fit error, and re-evaluation lowers one-shot error by roughly two thirds. Reduction is $100\times(1-\mathrm{returned\ RMSE}/\mathrm{one\mbox{-}shot\ RMSE})$; negative values indicate worsening.

\begin{table}[!htbp]
\centering
\caption{Nonlinear observed-factor benchmark with fifty assets. RMSEs use the on-policy law; parameters count fitted adjoint-field coefficients.}
\label{tab:nonlinear}
\footnotesize
\setlength{\tabcolsep}{4pt}
\begin{tabular}{llrrrr}
\toprule
Factors & Representation & Parameters & \makecell{Oracle-fit\\RMSE} & \makecell{Returned\\RMSE} & \makecell{One-shot\\reduction}\\
\midrule
3 & Quadratic & 270 & $5.658\times10^{-3}$ & $5.852\times10^{-3}$ & $-1.4\%$\\
3 & Fixed tanh features & $1{,}080$ & $4.539\times10^{-5}$ & $1.264\times10^{-3}$ & $68.1\%$\\
5 & Quadratic & 945 & $7.499\times10^{-3}$ & $8.256\times10^{-3}$ & $1.7\%$\\
5 & Fixed tanh features & $2{,}970$ & $6.223\times10^{-5}$ & $2.089\times10^{-3}$ & $67.7\%$\\
\bottomrule
\end{tabular}
\end{table}

\subsection{Flagship ablations and coverage}

Sampling and representation ablations use the legacy $1/6$-broad mechanism audit (\cref{fig:flagship-ablation}).

\begin{figure}[!htbp]
\centering
\includegraphics[width=0.60\textwidth]{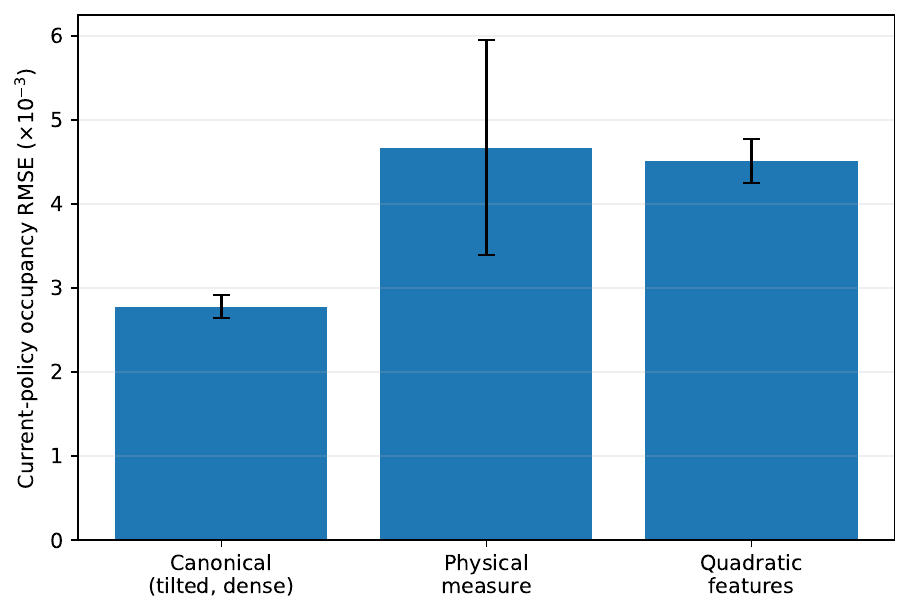}
\caption{Sampling and representation ablations: on-policy RMSE in the legacy $1/6$-broad mechanism audit. Error bars show one standard deviation across three seeds.}
\label{fig:flagship-ablation}
\end{figure}

The broad fraction specifies the share of restart-point anchors drawn from the broad stationary factor law; the remaining anchors follow the transient factor-occupancy law. Thus $50$--$50$ denotes an equal mixture. In the separate coverage audit, broad anchors and evaluation rescale stationary factor standard deviations by $1.8$. Tail-only evaluation uses a $2.2$ rescaling and excludes $\max_j|z_j/\varsigma_j|<1.5$, where $\varsigma_j$ is the stationary marginal standard deviation. The coverage audit uses standard-budget (core) and high-precision designs with respectively 128/256 anchors per time knot and 1,024/2,048 paths per anchor. Without broad anchors, standard-budget on-policy RMSE is about $0.244$. The pre-specified standard-budget broad-fraction sweep $\{0,1/6,1/3,1/2\}$ retains candidates within $5\%$ of the best on-policy RMSE, minimizes tail-only RMSE, and breaks ties by broad RMSE, selecting $1/2$. In the high-precision follow-up, moving from $1/6$ to $1/2$ lowers on-policy RMSE from $2.955\times10^{-3}$ to $2.758\times10^{-3}$ and broad/tail-only RMSE by about $41\%/44\%$. With the selected $50$--$50$ design, re-evaluation beats pooling in all three seeds on-policy at the standard and high-precision budgets, and on broad evaluation only at high precision; pooling wins tail-only at both budgets.

\FloatBarrier

\section*{Acknowledgements}
Jeonggyu Huh received financial support from the National Research Foundation of Korea\\
(No.~\mbox{RS-2025-00562904}).

\section*{Declarations}
\noindent\textbf{Code availability.} The reproducibility code and archived numerical summaries for this study are publicly available at the following GitHub repository:\par
\noindent\url{https://github.com/huhjeonggyu/self-consistent-adjoint-policy-iteration}

\noindent\textbf{Competing interests.} The authors declare no competing interests.

{\small
\sloppy
\setlength{\bibsep}{0.15em}
\bibliographystyle{spmpsci}
\bibliography{references}
}

\end{document}